\documentclass[11pt,a4paper]{article}

\usepackage{fullpage}
\usepackage{amsmath,amssymb}

\usepackage{mathtools}

\usepackage{bbm}

\usepackage{amsopn,amsthm}

\usepackage{subcaption}

\usepackage{graphicx} 

\usepackage{mathrsfs} 

\usepackage{hyperref}

\hypersetup{colorlinks=true, urlcolor=blue, citecolor=blue, linkcolor=blue}

\newtheorem{thm}{Theorem}[section]
\newtheorem{lem}[thm]{Lemma}

\newtheorem{proposition}[thm]{Proposition}
\newtheorem{cor}[thm]{Corollary}

\theoremstyle{remark}
\newtheorem{rem}[thm]{Remark}
\newtheorem{example}[thm]{Example}
\newcommand{\bone}{\mathbbm{1}}

\newcommand{\Lg}{\widehat{g}}
\newcommand{\Lu}{\widehat{u}}
\newcommand{\Lv}{\widehat{v}}
\newcommand{\Lw}{\widehat{w}}
\newcommand{\Ltw}{\widetilde{w}}

\newcommand{\field}[1]{{\mathbb{#1}}}
\newcommand{\C}{\field{C}}
\newcommand{\N}{\field{N}}
\newcommand{\R}{\field{R}}
\newcommand{\Fcal}{\mathcal{F}}
\newcommand{\Pcal}{\mathcal{P}}

\newcommand{\bs}{\boldsymbol} 

\newcommand{\bfd}{{\bs d}}
\newcommand{\bfg}{{\bs g}}
\newcommand{\bfh}{{\bs h}}
\newcommand{\bfp}{{\bs p}} 
\newcommand{\bfx}{{\bs x}}
\newcommand{\bfy}{{\bs y}} 
\newcommand{\bfz}{{\bs z}}
\newcommand{\bfP}{{\bs P}}
\newcommand{\bfnu}{{\bs\nu}}
\newcommand{\bfzeta}{{\bs\zeta}}

\newcommand{\dx}{\, \dif \bfx}
\newcommand{\ds}{\, \dif s}
\newcommand{\dt}{\, \dif t}
\newcommand{\dtheta}{\, \dif \theta}

\newcommand{\rmi}{\mathrm{i}} 
\newcommand{\rme}{\mathrm{e}}

\newcommand{\boldsymb}{ }
\newcommand{\cqb}{\boldsymb b}
\newcommand{\cqg}{\boldsymb g}
\newcommand{\cqC}{\boldsymb C}
\newcommand{\cqK}{\boldsymb K}
\newcommand{\cqH}{\boldsymb H}
\newcommand{\cqL}{\boldsymb L}
\newcommand{\cqW}{\boldsymb W}
\newcommand{\cqX}{\boldsymb X}
\newcommand{\cqY}{\boldsymb Y}

\newcommand{\vtn}{\underline{t_n}}

\DeclareMathOperator*{\argmin}{arg\,min}
\DeclareMathOperator{\acosh}{acosh}

\renewcommand{\div}{\mathrm{div}}

\newcommand{\tri}{\bigtriangleup}

\newcommand{\ui}{u^i}

\newcommand{\Lap}{\mathcal{L}}

\numberwithin{equation}{section}

\DeclareMathOperator{\dif}{d\!}  
\DeclareMathOperator{\real}{Re}

\DeclareMathOperator{\Sop}{\mathrm{S}}
\DeclareMathOperator{\Dop}{\mathrm{D}}
\DeclareMathOperator{\Vop}{\mathrm{V}}

\begin{document}

\title{The temporal domain derivative in inverse acoustic obstacle scattering} 
\author{Marvin
  Kn\"oller\footnote{Department of Mathematics and Statistics, University of Helsinki,
    00014 Helsinki, Finland
    {\tt marvin.knoller@helsinki.fi}} \; and\;
 J\"org Nick\footnote{Seminar f\"ur Angewandte Mathematik, ETH Z\"urich, CH-8092 Z\"urich, Switzerland {\tt  joerg.nick@math.ethz.ch}}}

\maketitle

\begin{abstract}
 This work describes and analyzes the domain derivative for a time-dependent acoustic scattering problem.
We study the nonlinear operator that maps a sound-soft scattering object to the solution of the time-dependent wave equation evaluated at a finite number of points away from the obstacle.
The Fr\'echet derivative of this operator with respect to variations of the scatterer coincides with point evaluations of the temporal domain derivative. The latter is the solution to another time-dependent scattering problem, for which a well-posedness result is shown under sufficient temporal regularity of the incoming wave.
Applying convolution quadrature to this scattering problem gives a stable and provably convergent semi-discretization in time, provided that the incoming wave is sufficient regular. 
Using the discrete domain derivative in a Gauss--Newton method, we describe an efficient algorithm to reconstruct the boundary of an unknown scattering object from time domain measurements in a few points away from the boundary. Numerical examples for the acoustic wave equation in two dimensions demonstrate the performance of the method. 
 \end{abstract}

{
\small\noindent
  Mathematics subject classifications (MSC2020): 
  35R30, 
 65N21 
  \\\noindent
Short title: The temporal domain derivative in inverse scattering
}

\section{Introduction}
Understanding the effects of boundary perturbations on measurements of a solution to a partial differential equation is a key difficulty in many applications, such as inverse scattering and shape optimization. When the magnitude of boundary perturbation becomes small, their impacts are, at leading order, described by the domain derivative.
The domain derivative is typically 
expressed as a solution to a partial differential equation that leaves the original equation in the relevant domains unaffected, whereas boundary or transmission conditions change to a term that depends on the solution of the unperturbed problem and the
perturbation's normal component.
In the context of time-harmonic wave phenomena, an extensive literature studying the domain derivative has been developed in the last decades, starting from the pioneering work about sound-soft scattering for the Helmholtz equation in \cite{Kirsch93}. 
Afterwards, different boundary conditions for the scatterer were studied, such as e.g.\@ Robin boundary conditions in \cite{Hett95},\cite{Het98e}, nonlinear impedance boundary conditions in \cite{FiHe24} or transmission conditions in \cite{Hett95, Hett99}. 
An efficient reconstruction method for three-dimensional sound-soft obstacles based on domain derivatives has been proposed in \cite{HaHo07}.
Domain derivatives for the time-harmonic Maxwell's equations featuring the perfect conductor boundary condition as well as transmission boundary conditions were studied in \cite{Hag19, Hagetal19, Het12}.
A different approach in developing domain derivatives for scattering problems is to represent waves using single and double layer potentials and to derive these representations with respect to variations of the domain.
For the time-harmonic Helmholtz equation this was done in \cite{HadKre04, HoScho98, Pott94, Pott96}, for time-harmonic Maxwell's equations see \cite{CosLou12, HadKre04, Pott96b}.
\begin{figure}[tbp]
    \centering
    \begin{minipage}{0.34\textwidth}
        \centering
        \includegraphics[width=\linewidth]{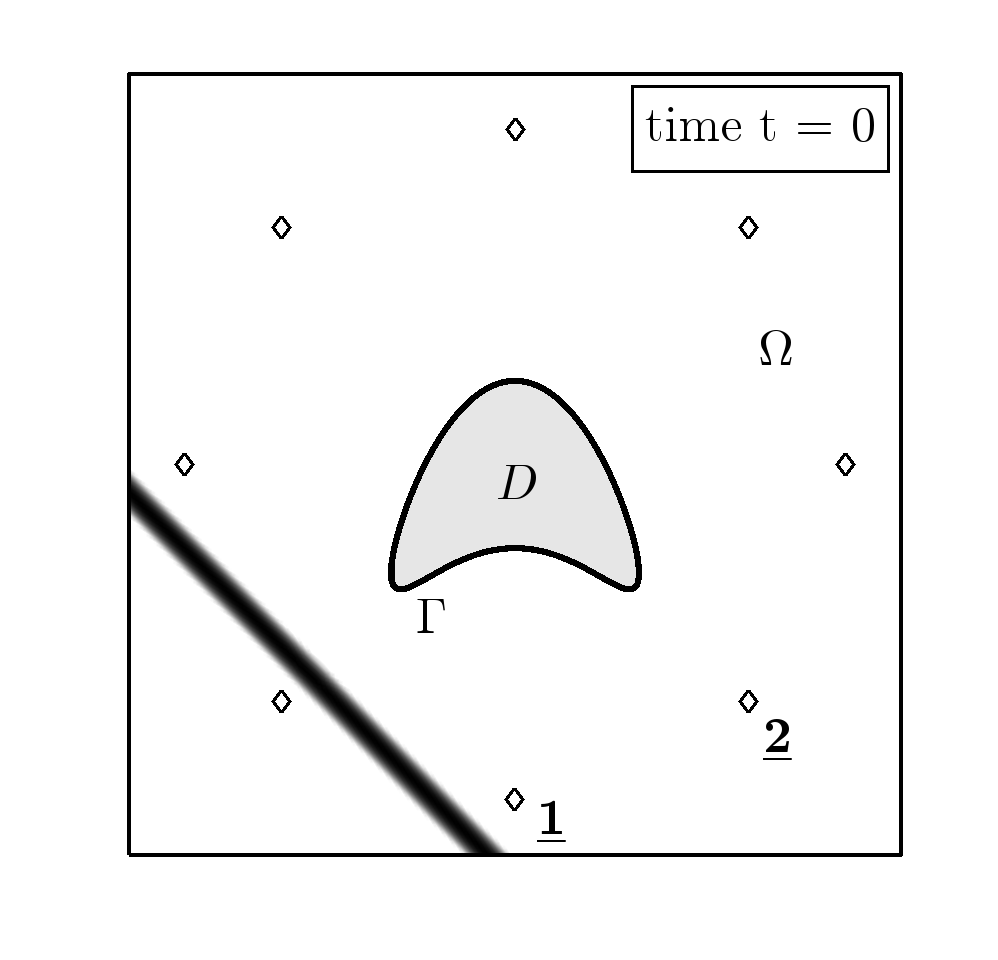}
    \end{minipage}
    \begin{minipage}{0.34\textwidth}
        \centering
        \includegraphics[width=\linewidth]{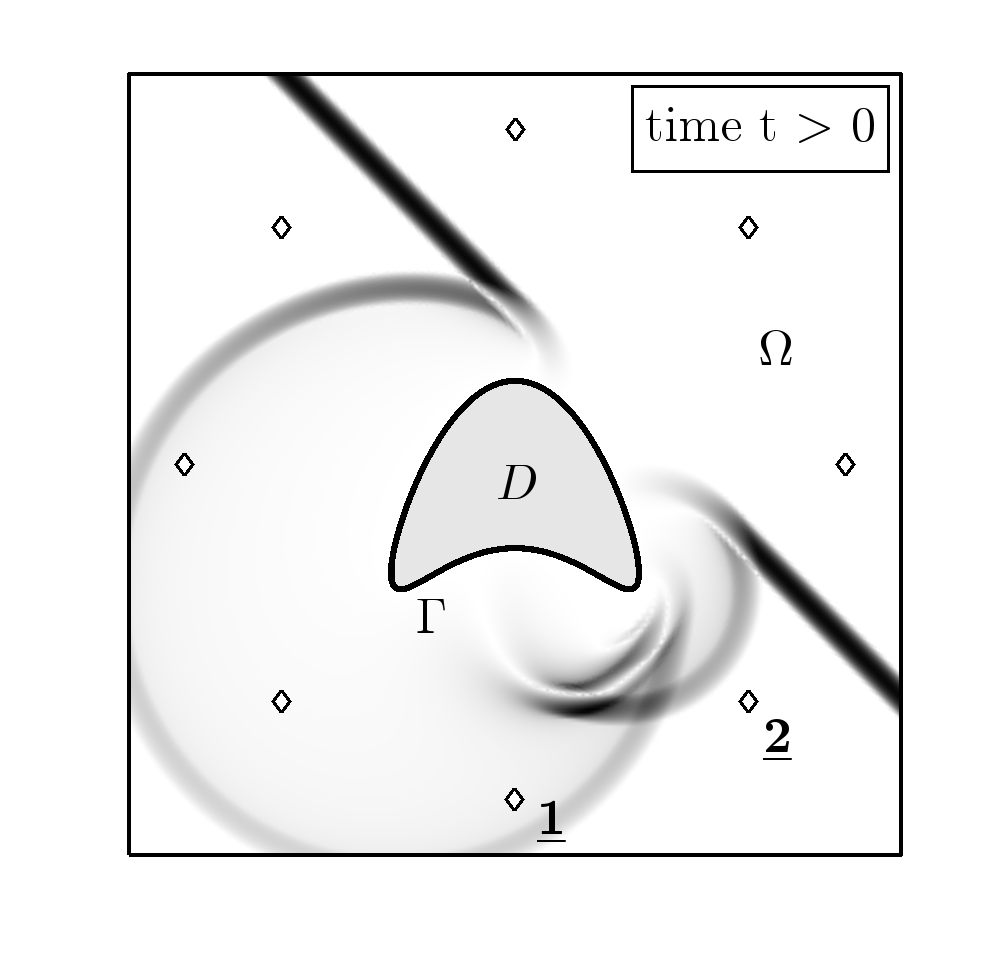}
    \end{minipage}
    \begin{minipage}{0.305\textwidth}
        \centering
        \includegraphics[width=\linewidth]{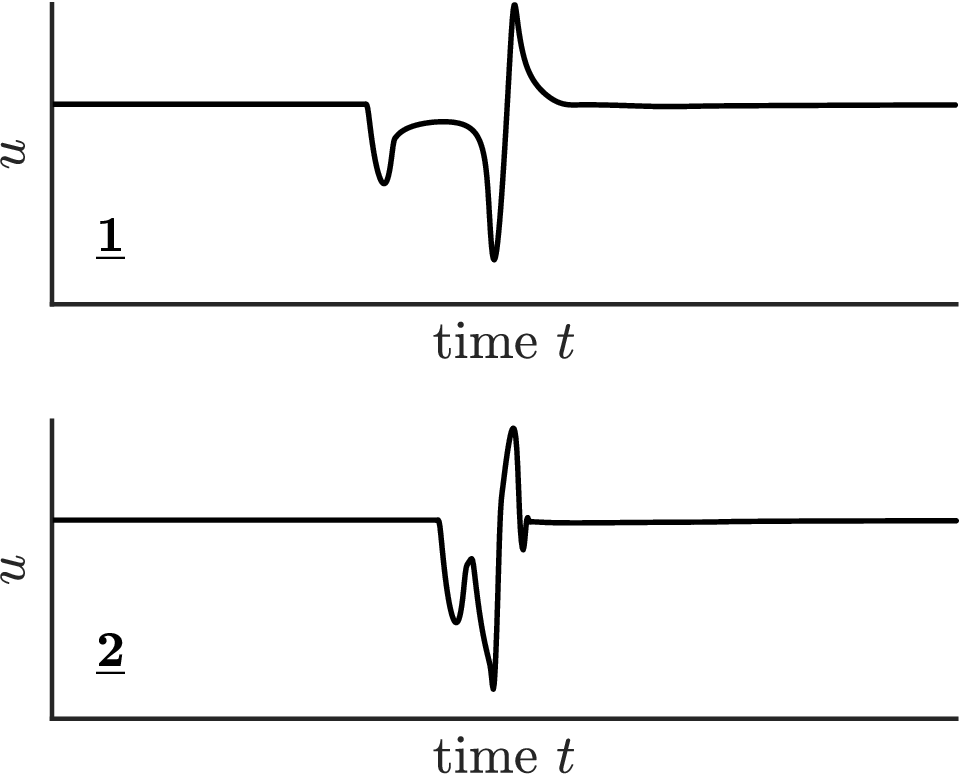}
    \end{minipage}
    \caption{Scattering from a kite-shaped obstacle $D$. In the left and middle plot the solid black strip represents the support of the incoming wave $\ui$. Diamonds represent receivers that detect the scattered wave $u$. Visualizations of $u$ at the receivers $1$ and $2$ are found in the right plot.}
    \label{fig:sketch}
\end{figure}
In this paper, we study a well-posedness result of the time-dependent domain derivative related to the exterior sound-soft scattering problem. 
In order to introduce this problem, let $\Omega = \R^d \setminus \overline{D}$ for the dimensions $d\in \{2,3\}$ denote the complement of a bounded scatterer $D$, having a $C^2$ boundary $\Gamma =\partial\Omega $. 
The acoustic wave equation in the exterior connected domain $\Omega$ in a finite time interval $[0,T]$ with final time $T>0$ reads
\begin{align}\label{eq:scatwave-1}
	\partial_t^2 u - \Delta u \, = \, 0 \quad &\text{in } \Omega \times [0,T] \, .
\end{align}
Let $\ui$ be an incident wave, which fulfills the acoustic wave equation in the entire space $\R^d$ and for all times. Furthermore, the support of $\ui$ at time $t=0$ shall not intersect with the boundary of the obstacle $\Gamma$. 
This condition guarantees vanishing initial conditions for the scattered wave $u$ and its derivatives.
The boundary condition formulated for the scattered wave $u$ in presence of a sound-soft scattering object reads
\begin{align}\label{eq:scatwave-2}
	u(\bfx,t) \, = \, -\ui(\bfx,t) \quad &\text{on } \Gamma \times [0,T]\, .	
\end{align}
A sketch of the scattering problem is found in Figure~\ref{fig:sketch}. The initial support of the incoming wave $\ui$ is visualized as the solid black strip that is found in the left plot. As time proceeds, $\ui$ impinges on the obstacle and produces the scattered wave $u$ seen in the middle plot. Receivers that measure the scattered wave are indicated by diamonds. Visualizations of the measurements at the receivers $1$ and $2$ are found in the right plot.

We understand the direct problem as the task of computing the scattered wave $u$ at selected points in $\Omega$, given the incident wave $\ui$ and the obstacle $D$.
Regarding Figure~\ref{fig:sketch}, this refers to the determination of $u$ at the receivers, as depicted in the right plot.
In contrast, the inverse problem that we study consists of the challenge to reconstruct the scattering object $D$ from given measurements of the scattered wave $u$ at some receivers in $\Omega$ and given knowledge on the incoming wave $\ui$. 

For shape identification methods in the time-domain,
qualitative methods for the wave equation such as the linear sampling and factorization method were studied in \cite{CakHadLech19, CakMonSel21, ChenHadLechMon10, GuoMonCol13, HadLechMar14, HadLiu20}.
For time-dependent Maxwell's equations a linear sampling method has been studied in \cite{LaMonSel22}.
Recently, efforts have been made to use iterative shape reconstruction methods for time-dependent scattering problems.
In \cite{ZhaoDongMa21} and \cite{ZhaoDongMa22} shape reconstructions are considered for the acoustic wave equation and the elastic wave equation, respectively. 
The methods presented in these works use a discrete number of frequencies, which are generated from the convolution quadrature method to formulate several frequency-domain scattering problems.
Afterwards, frequency-domain boundary integral operators are linearized with respect to the two-dimensional boundary's parametrization to construct an iterative domain reconstruction method. At its core, the initial time-domain scattering problem is replaced by several frequency-domain problems.

In order to derive the temporal domain derivative we here apply the Laplace transform to the wave equation \eqref{eq:scatwave-1} and utilize results similar to those, which were established for the time-harmonic Helmholtz equation in \cite{Kirsch93}.
Tracking dependencies on powers of the complex-valued wave number and the norm of the incoming wave, an inverse Laplace transform provides requirements on the time regularity of the incoming time-dependent wave necessary to guarantee the existence of the temporal domain derivative.
Moreover, we conclude that the temporal domain derivative is the unique solution to a time-dependent scattering problem featuring a boundary condition inherited from the Laplace domain.
Similar formulations for this time-dependent problem are also found in the literature as shape derivatives to wave equations (see \cite{CagZol98, CagZol99, SokZol92}).
Discretizing this scattering problem yields numerical schemes for the temporal domain derivative. 
We discretize this scattering problem in time by the convolution quadrature method based on Runge--Kutta multistage methods. 

We understand the inverse shape reconstruction problem as the determination of the boundary $\Gamma$ from several point evaluations of the corresponding time-dependent scattered wave.
The temporal domain derivative provides the Fr\'echet derivative of the functional corresponding to point evaluations of the scattered wave.
Moreover, by using the framework of temporal convolution operators and boundary integral equations, we obtain a well-posedness result for the discretization in time for this Fr\'echet derivative. 
Finally, we set up a Gau\ss--Newton method for the iterative reconstruction of a two-dimensional sound-soft scattering object and provide numerical examples.

The paper is structured in the following way.
In Section~2, we recall the Laplace domain setting of the scattering problem. 
Crucial frequency explicit estimates for the solution to the scattering problem are shown in Proposition~\ref{prop:H2-bound} and subsequently for the domain derivative in Proposition~\ref{prop:th-dom-deriv-bounds} and Proposition~\ref{prop:th-domain-derivative}. In Section~3, we use these results to obtain well-posedness results and bounds for the temporal domain derivative, formulated in Theorem~\ref{th:td-dd-bound}. Subsequently in Section~4, a convolution quadrature time discretization based on a Runge--Kutta time stepping method is employed, which yields a time-discrete approximation to the domain derivative. 
General convolution quadrature approximation results and the bounds for the Laplace domain problem of Section~2 give the error estimates formulated in Theorem~\ref{th:cq-conv}.
In the last section, we use the temporal domain derivative in a Gau\ss--Newton method for reconstructing two-dimensional scattering objects. Numerical experiments demonstrate the feasibility of the method for different configurations of the scattering problem.

\section{Frequency dependent estimates in the Laplace domain}
The wave equation \eqref{eq:scatwave-1} together with the Dirichlet boundary condition \eqref{eq:scatwave-2} has a unique solution. The time and space regularity of $u$ depends on the regularity of the incoming wave on the boundary. Precise formulations of these regularities require adequate Sobolev spaces and thus, discussions on the well-posedness of \eqref{eq:scatwave-1}-\eqref{eq:scatwave-2} are found in the next chapter, more precisely, in Corollary~\ref{cor:wp}.

Our study starts in the Laplace domain, in which we derive frequency explicit results applicable for the Helmholtz equation - the frequency domain pendant to the acoustic wave equation.
For any causal function $g$, we define the Laplace transform $\mathcal{L}$ as
\begin{align*}
\mathcal{L}\{g\} (s) \, := \, \int_0^\infty \rme^{-st} g(t) \dt \, \quad \text{for any } s \in \C_+ \, := \, \{z \in \C \, | \, \real z > 0 \} \, .
\end{align*}
The Laplace transform of temporal solutions of the acoustic wave equation constitute a solution of the Helmholtz equation in the exterior domain $\Omega = \R^d\setminus \overline{D}$. For $s\in\mathbb C$ with $\real s > 0 $, the formulation for the 
scattered field $\Lu=\Lap\{u\}$ in the Laplace domain reads
\begin{subequations}\label{eq:HHeq}
\begin{align}
 s^2 \Lu -\Delta\Lu \,=\, 0  \quad &\text{in } \Omega \label{eq:HHeq1} \, ,
 \\
 \Lu \, = \, -\Lu^i \quad &\text{on } \Gamma\, ,  \label{eq:HHeq2}
\end{align}
 \end{subequations}
where $\Lu^i = \Lap\{\ui\}$ is the Laplace transform of the incident field. Since $\real s>0$, it suffices to demand an integrability condition for the uniqueness of the formulation, such as $\widehat u(s)\in H^1(\Omega).$ Alternatively, the strong formulation can be completed by an appropriate asymptotic condition for $|\bfx|\rightarrow\infty$. Both $\Lu$ and $\Lu^i$ depend on $s\in \C_+$ and $\bfx \in \Omega$. 
\begin{rem}[Notational convention]
    Throughout the manuscript, we denote functions that are Laplace transforms of a temporal function of interest with a hat (such as $\widehat u$) and omit the omnipresent argument $s$. In contrast, operators in the Laplace domain are denoted by capital letters and are accompanied by the Laplace domain argument $s$, such as $\Sop(s)$, as is typical in view of the Heaviside notation of operational calculus.
\end{rem}
\subsection{Potential and boundary integral operators} 
We recall the potential and boundary operators associated to the Helmholtz problem, as described for example in \cite{BanSay22, Say16}.
First, we introduce the fundamental solution $\Phi_s(\bfx,\bfy)$ given by
\begin{align*}
\Phi_s(\bfx,\bfy) \, := \, 
\begin{dcases*}
\frac{\rmi}{4} H_0^{(1)}(\rmi s |\bfx - \bfy|)\, , \quad &d = 2\,, \\
\frac{\rme^{-s |\bfx-\bfy|}}{4\pi |\bfx-\bfy|} \, , \quad &d = 3\,,
\end{dcases*}
\qquad \text{for } \bfx, \bfy \in \R^d\, ,\; \bfx \neq \bfy \, , 
\end{align*}
where $H_0^{(1)}$ denotes the Hankel function of the first kind and order $0$.
The single-layer potential is then defined by $\Sop(s): H^{-1/2}( \Gamma) \to H_\Delta^{1}(\Omega)$ with
\begin{align*}
	(\Sop(s)\varphi)(\bfx) \, := \, \int_{\Gamma} \Phi_s(\bfx,\bfy) \varphi(\bfy) \ds(\bfy) \, , \quad \bfx \in \Omega \, ,
\end{align*}
where $ H_\Delta^{1}(\Omega) = \{ u \in H^1(\Omega) \, : \, \Delta u \in L^2(\Omega) \}$. The double-layer potential is defined as \newline $\Dop(s) :H^{1/2}(\Gamma) \to H_\Delta^1(\Omega)$ with
\begin{align*}
	(\Dop(s)\psi)(\bfx) \, := \, \int_{\Gamma} \partial_{\bfnu_{\bfy}}
	\Phi_s(\bfx,\bfy) \psi(\bfy) \ds(\bfy) \, , \quad \bfx \in \Omega\, ,
\end{align*}
where $\bfnu_{\bfy}$ denotes the exterior unit normal vector at $\bfy\in \Gamma$. 

Any solution $\widehat u$ to the Helmholtz equation \eqref{eq:HHeq1} fulfills the representation formula 
\begin{align}\label{eq:representation}
	\Lu(\bfx)
\, = \,  - \left(\Sop(s) \partial_\bfnu \Lu\right)(\bfx) +\left(\Dop(s)  \gamma \Lu\right)(\bfx) \, , \quad \bfx \in \Omega \, ,
\end{align}
which is the Laplace domain pendant to Kirchhoff's formula (see \cite[Sec.\@ 4.8]{BanSay22}),
where $\gamma \Lu$ and $\partial_\bfnu \Lu$ denote the trace from $H^1(\Omega)$ to $H^{1/2}(\Gamma)$ and the normal trace from $H_\Delta^1(\Omega)$ to $H^{-1/2}(\Gamma)$, respectively.
We define $H_0^1(\Omega) = \{u \in H^1(\Omega) \, : \, \gamma u = 0\}$.
The boundary integral operator associated to $\Sop(s)$ is given by $\Vop(s): H^{-1/2}(\Gamma) \to H^{1/2}( \Gamma)$ with
\begin{align*}
	\left(\Vop(s)\varphi\right)(\bfx) \, &:= \, \int_{\Gamma} \Phi_s(\bfx,\bfy) \varphi(\bfy) \ds(\bfy)\, , \quad \bfx \in \Gamma \, .
\end{align*}
Problem \eqref{eq:HHeq} is uniquely solved by a single-layer approach as outlined below. Suppose that 
	$\Lu(\bfx) = (\Sop(s)\widehat{\varphi})(\bfx)$ for $\bfx \in \Omega$ and
for a density $\widehat{\varphi} \in H^{-1/2}(\Gamma)$ yet to be determined.
Applying the trace operator $\gamma$ to both sides and using the boundary condition \eqref{eq:HHeq2} gives that
	$-\Lu^i(\bfx) \,=\, (\Vop(s)\widehat{\varphi})(\bfx)$ for $\bfx \in \Gamma$. 
By the coercivity of the single-layer operator, it follows that $\Vop^{-1}(s): H^{1/2}(\Gamma) \to H^{-1/2}(\Gamma)$ is bounded and 
\begin{align}\label{bound-V-m-1}
	\left\Vert \Vop^{-1}(s) \right\Vert_{H^{-1/2}(\Gamma)\leftarrow H^{1/2}(\Gamma)} \, \leq \, \frac{C_\sigma}{\real s} |s|^2\, , \quad \real s \geq \sigma > 0 \, 
\end{align}
(see \cite[Thm.\@ 4.6]{BanSay22}).
Thus, the density $\widehat{\varphi}$ may be represented via $\widehat{\varphi} = -\Vop^{-1}(s)\gamma\Lu^i$. We obtain the solution $\Lu$ by using this density in the representation from above, i.e.
\begin{align}\label{eq:SV-1-th}
	\Lu(\bfx) \,=\, -(\Sop(s)\Vop^{-1}(s)\gamma\Lu^i)(\bfx)\, , \quad \bfx \in \Omega \, .
\end{align}
In our special case, in which we assume $\Lu^i|_D$ to be a solution to the Helmholtz equation, i.e., $ s^2 \Lu^i - \Delta\Lu^i  = 0$ in $D$, the density $\widehat{\varphi}$ possesses a special representation, namely $\widehat{\varphi} = -\partial_\bfnu(\Lu + \Lu^i)$. To see this, we first apply Green's formula to $\Lu^i$ and afterwards to $\Phi_s(\bfx,\cdot)$, what shows that 
\begin{align}\label{eq:rep2}
	- (\Sop(s) \partial_\bfnu \Lu^i)(\bfx) +(\Dop(s)  \gamma \Lu^i)(\bfx)
\, = \, 0  \, , \quad \bfx \in \Omega \, 
\end{align}
(cf. \cite[Thm.\@ 5.39]{KiHe15}). Adding \eqref{eq:representation} and \eqref{eq:rep2} and using \eqref{eq:HHeq2} yields that
\begin{align}\label{eq:reptotal}
-\left( \Sop(s) \left( \partial_\bfnu (\Lu + \Lu^i)\right)\right)(\bfx) \, = \, \Lu(\bfx) \quad \text{for } \bfx \in \Omega\,.
\end{align}
Applying the trace operator $\gamma$ to both sides of \eqref{eq:reptotal} yields that 
\begin{align}\label{eq:obtNeumann}
\Vop(s) \left( \partial_\bfnu (\Lu + \Lu^i)\right) \, = \, \gamma \Lu^i \, \quad\text{or equivalently, } \;   -\partial_\bfnu (\Lu + \Lu^i) \, = \, -\Vop^{-1}(s)\gamma \Lu^i \, = \, \widehat{\varphi} \, .
\end{align}
For $\Lu$ as defined in \eqref{eq:SV-1-th}, it holds that (see also \cite[Prop.\@~1]{Bam86a})
\begin{align}\label{eq:wpin}
	\left\Vert \Lu \right\Vert_{|s|,\Omega} \, := \, \left(\left\Vert \nabla \Lu \right\Vert_{L^2(\Omega)}^2 +  \left\Vert s\Lu \right\Vert_{L^2(\Omega)}^2\right)^{1/2}  \, \leq \, \frac{C_\sigma}{\real s} |s|^{3/2} \left\Vert \gamma\Lu^i \right\Vert_{H^{1/2}(\Gamma)} \,, \quad \real s \geq \sigma > 0\, .
\end{align}
The norm $\Vert \cdot \Vert_{|s|,\Omega}$ is equivalent to the natural norm on $H^1(\Omega)$, since (see also \cite[Sec.\@ 4.4]{BanSay22})
\begin{align}\label{eq:equiv}
\min\{1,|s|\} \Vert \Lu \Vert_{H^1(\Omega)} \, \leq \, \Vert \Lu \Vert_{|s|,\Omega} \, \leq \, \max\{1,|s|\} \Vert \Lu \Vert_{H^1(\Omega)} \, .
\end{align}
The scalar product that induces $\Vert \cdot \Vert_{|s|,\Omega}$ is given by
\begin{align}\label{eq:skpS}
\langle \Lu,\Lv \rangle_{|s|,\Omega} \, := \, \langle \nabla \Lu, \nabla \Lv \rangle_{L^2(\Omega)} + |s|^2\langle \Lu, \Lv \rangle_{L^2(\Omega)} \, .
\end{align}
In this work we require pointwise bounds on the scattered wave $\Lu$, which follow from applying a duality argument directly to the integral form of the potential operators. This is found in \cite[Lem.\@ 7]{BLM11} for the single layer operator in $d=3$. The proof is directly applicable to the double-layer potential for $d=3$. 
For dimension $d=2$ the proof can be done similarly by noting that $(H_0^{(1)})'= -H_1^{(1)}$ and using the estimate (see e.g.\@ \cite[Eq.\@ 9.2.3]{AbSte64})
\begin{align*}
|H_\nu^{(1)}(\rmi s |\bfx-\bfy|)| \, \leq \, C  |s|^{-1/2}|\bfx-\bfy|^{-1/2} \rme^{-\real s |\bfx-\bfy|} \quad \text{for } s \in \C_+ \text{ with } \real s \geq \sigma\, \; \text{and fixed } \nu\, ,
\end{align*}
with $\bfx \in \Omega$ and $\bfy \in D$.
We therefore obtain the estimates
\begin{subequations} \label{eq:pwbound}
\begin{align}
\left| (\Sop(s)\varphi) (\bfx) \right| &\, \le \, C(\sigma,\mathrm{dist}(\bfx,\Gamma))|\left|s\right|^{{(d-1)}/{2}} e^{-\mathrm{dist}(\bfx,\Gamma)\real s} \left\Vert \varphi \right\Vert_{H^{-1/2}(\Gamma)} \, ,
&&\varphi \in H^{-1/2}(\Gamma)\, ,\label{eq:pwbound1}\\
 \left| (\Dop(s)\psi) (\bfx) \right| & \, \le \, C(\sigma,\mathrm{dist}(\bfx,\Gamma))\left|s\right|^{{(d-1)}/{2}} e^{-\mathrm{dist}(\bfx,\Gamma)\real s} \left\Vert \psi \right\Vert_{H^{1/2}(\Gamma)} \, ,  &&\psi \in H^{1/2}(\Gamma)\, ,
\end{align}
\end{subequations}
where the constant $C(\sigma,\mathrm{dist}(\bfx,\Gamma))$ depends on $\sigma^{-1}$ and on $\mathrm{dist}(\bfx,\Gamma)^{-(d-1)/2}$ for both dimensions $d=2,3$.

We study the frequency dependence of $H^2$ solutions to $s^2 \Lu - \Delta \Lu =0$ in the next proposition. 
The proof can be done as the proof of \cite[Thm.\@ 8.8, Thm.\@ 8.12]{GilTru77} by additionally tracking dependencies on powers of $|s|$.
\begin{proposition}\label{prop:H2-bound}
Let $D$ be a bounded $C^2$-domain and let $\Omega = \R^d\setminus \overline{D}$ with boundary $\partial \Omega=\Gamma$. Further, let $\Lg \in H^{3/2}(\Gamma)$ and let $\mathcal{O} \in \{D,\Omega\} $.
Then, the unique solution $\Lv\in H^1(\mathcal{O})$ of
\begin{subequations}\label{eq:HHeqg}
\begin{align}
 s^2 \Lv-\Delta\Lv\,=\, 0  \quad &\text{in } \mathcal{O} \, ,
 \\
 \Lv \, = \, \Lg \quad &\text{on } \Gamma\, 
\end{align}
\end{subequations}
is also in $H^2(\mathcal{O})$ and satisfies
\begin{align}\label{eq:propRes}
\Vert \Lv \Vert_{H^2(\mathcal{O})} \, \leq \, C_\sigma \frac{|s|^{5/2}}{(\real s)^{1/2}} \Vert \Lg \Vert_{H^{3/2}(\Gamma)} \, , \quad \real s \geq \sigma > 0\,
\end{align}
with a constant $C_\sigma>0$ that does not depend on the frequency $s$.
\end{proposition}
\begin{proof}
We only give a sketch of the proof.
Due to our assumption that $D$ is a $C^2$ domain, the trace operator $\gamma : H^2(\mathcal{O}) \to H^{3/2}(\Gamma)$ has a bounded right inverse denoted by $\eta : H^{3/2}(\Gamma) \to H^2(\mathcal{O})$. Since we assume $\Lg$ to be in $H^{3/2}(\Gamma)$, the function $\varphi := \eta \Lg\in H^2(\mathcal{O})$ is well-defined. Moreover, we define the function $\widehat{w}:= \widehat{v} + \varphi \in H_0^1(\mathcal{O})$, where $\Lv \in H^1(\mathcal{O})$ is the unique weak solution to \eqref{eq:HHeqg}.
Proceeding as in the proofs of \cite[Thm.\@ 8.8, Thm.\@ 8.12]{GilTru77} but tracking dependencies on powers of $|s|$ yields that $\widehat{w} \in H^2(\mathcal{O})$ and 
\begin{align}\label{eq:giltru}
\Vert \widehat{w}\Vert_{H^2(\mathcal{O})} \, \leq \, C
\left(\max\{1,|s|\} \Vert \widehat{w}\Vert_{|s|,\mathcal{O}}+ 
\max\{1,|s|^2\}\Vert \varphi \Vert_{H^2(\mathcal{O})}\right)\,
\end{align}
with a constant $C>0$ that does not depend on the frequency $s$. The $s$-dependent norm on the right hand side of
\eqref{eq:giltru} is defined in \eqref{eq:wpin}. 
Using the estimate in \eqref{eq:wpin} with $\gamma \Lu^i$ replaced by $\Lg$ and the boundedness of $\eta$, i.e., $\Vert \varphi \Vert_{H^2(\mathcal{O})} \leq C\Vert \Lg \Vert_{H^{3/2}(\Gamma)}$ yields the estimate \eqref{eq:propRes}.
\end{proof}

\subsection{The domain derivative in the Laplace domain}\label{sec:th-dom-deriv}
Throughout the paper, for two open and bounded sets $D_1,D_2 \subset \R^d$ the notation $D_1 \subset \subset D_2$ means that $D_1 \subset D_2$ and $\mathrm{dist}(D_1,\partial D_2)>0$.
For the bounded $C^2$ domain $D$ we consider a vector field $\bfh \in C^1(\Gamma,\R^d)$ that is supposed to deform $\partial D$. 
Let $R>0$ be sufficiently large such that $D \subset\subset B_R(0)$. Then there exists an extension of $\bfh$ supported in a neighborhood of $\Gamma$, again denoted by $\bfh$, such that $\Vert \bfh \Vert_{C^1(B_R(0))} \leq C \Vert \bfh \Vert_{C^1(\Gamma)}$ for some $C>0$ (see e.g.\@ \cite[Thm.\@ 1.5]{Ho99}).
For sufficiently small $\Vert \bfh \Vert_{C^1(\Gamma)}$, the deformation
\begin{align}\label{eq:phi-diff}
	\varphi: B_R(0) \to B_R(0), \qquad \varphi(\bfx) \, := \, \bfx + \bfh(\bfx)
\end{align}
is a diffeomorphism and
maps $D$ to $D_{\bfh} := \varphi(D)$.
A variation of the boundary of the scattering object $\Gamma$ by $\bfh$ via $\Gamma_\bfh = (I+\bfh)\Gamma$ affects the solution of the Helmholtz equation \eqref{eq:HHeq} through the perturbation of its underlying geometry.
Let $\bfz_j \in \R^d$ for $j=1,\dots,M$ be some discrete observation points located away from all scattering objects under consideration.
For a fixed incoming wave $\Lu^i$, we define the nonlinear boundary-to-solution operator $\widehat{X}: \Gamma \mapsto (\Lu(\bfz_j))_{j=1,\dots,M}$, where $\Lu$ solves \eqref{eq:HHeq}.
By $D(\widehat{F}_\Gamma)$ we denote a neighborhood of the zero function in $C^1(\Gamma,\R^d)$ that is so small that
\begin{align}\label{eq:domainder}
\widehat{F}_\Gamma : D(\widehat{F}_\Gamma) \subset C^1(\Gamma, \R^d) \to \C^M\, , \qquad \widehat{F}_\Gamma(\bfh) \, := \, \widehat{X}(\Gamma_\bfh) \, 
\end{align}
is well-defined.
The next proposition, which is similar to \cite[Thm.\@ 2.1]{Kirsch93}, guarantees the existence of the Fr\'echet derivative $\widehat{F}_\Gamma'(0) : C^1(\Gamma,\R^d) \to \C^M$ and furthermore, provides a characterization in terms of a solution to the Helmholtz equation. The proof that we provide here follows the proofs of \cite[Thm. 2.1]{Kirsch93} and \cite[Ch.\@ 2.2]{Hett99}. Here, however, we perform estimates explicit in terms of powers of $|s|$ and the norm of the incident wave $\ui$ as we require this in our analysis later on. Moreover, we do not have to truncate the domain as in the proof of \cite[Thm. 2.1]{Kirsch93} since the functions that we have to deal with are globally in $H^1$.

\begin{proposition}\label{prop:th-dom-deriv-bounds}
	Let $D$ be a bounded $C^2$-domain, $\partial D = \Gamma$ and let $\Lu\in H^2(\Omega)$ be the unique solution of \eqref{eq:HHeq}. 
	Then, the operator $\widehat{F}_\Gamma$ from \eqref{eq:domainder} is Fr\'echet differentiable at zero and $\widehat{F}_\Gamma'(0)\bfh$ is given by the solution $\Lu'\in  H^1(\Omega)$ of
	\begin{align}\label{eq:HHUp}
		\begin{split}
			 s^2 \Lu'-\Delta \Lu' \,=\, 0  \quad &\text{in } \Omega \\
			\Lu' \, = \, -\left(\bfh \cdot \bfnu\right)\partial_\bfnu (\Lu+\Lu^i) \quad &\text{on } \Gamma,
		\end{split}
	\end{align}
	evaluated at $\bfz_j$, $j=1,\dots,M$, i.e.\@ $\widehat{F}'_{\Gamma}(0) \bfh = \left(\Lu'(\bfz_j)\right)_{j=1,\cdots,M}$. The function $\Lu'$ is called the domain derivative in the Laplace domain.
	Moreover, let $h_0>0$ and let $D_0\subset \R^d$ be a bounded domain such that $D_\bfh \subset \subset D_0$ for all $\Vert \bfh \Vert_{C^1(\Gamma)}\leq h_0$ and $\bfz_j \notin D_0$ for all $j=1,\dots,M$. Then, with $d_{\mathrm{min}} := \min\{\mathrm{dist}(\bfz_j,\partial D_0)\, : \, j=1,\dots,M\}$ it holds that
\begin{align}\label{eq:ddest}
\left| \widehat{F}_\Gamma(\bfh) - \widehat{F}_\Gamma(0) - \widehat{F}_\Gamma'(0) \bfh  \right|  \,
 \leq \, C(\sigma,d_{\mathrm{min}}) e^{-d_{\mathrm{min}}\real s}   \frac{|s|^{(d+5)/2}}{(\real s)^3} \left\Vert \Lu^i \right\Vert_{H^{1}(D_1\setminus \overline{D})} \Vert \bfh \Vert_{C^1(\Gamma)}^2\, ,
\end{align} 
for some bounded domain $D_1$ with $D_0\subset \subset D_1$, where the constant $C(\sigma,\xi)$ depends on $\sigma^{-1}$ and on $d_{\mathrm{min}}^{-(d-1)/2}$ for the dimensions $d=2,3$.
\end{proposition}
\begin{rem}
    We note that formulas as \eqref{eq:HHUp} are well-known in the literature surrounding shape derivatives (see e.g. \cite[Sec.\@ 3.1]{SokZol92} for the shape derivative for the Poisson equation).
\end{rem}
\begin{proof}
Let $\chi  \in C^\infty(\R^d)$ be a cut-off function such that $\chi = 1 $ in $D_0$ and $\chi = 0$ in $\R^d \setminus \overline{D_1}$.
	The weak formulation of \eqref{eq:HHeq} can be formulated as the task to find $\Lw= \Lu+\Lu^i\chi \in H_0^1(\Omega)$ such that $a(\Lw,v) = f(v)$ for all $v \in H_0^1(\Omega)$ where
	\begin{align}\label{eq:andf}
	a(\Lw,v) \,:=\, \int_{\Omega} \nabla \Lw \cdot \overline{\nabla v} + s^2 \Lw \overline{v} \dx \, \quad \text{and} \quad  
	f(v) \, := \, \int_{ D_1\setminus \overline{D_0}} \nabla \chi \cdot ( \Lu^i\overline{\nabla v} - \overline{v} \nabla \Lu^i) \dx\, .
	\end{align}
    We start with a basic bound on $\Lw$. Testing $a(\Lw,v) = f(v)$ with $v=s\Lw$ and taking the real part on both sides yields
    \begin{align*}
        \real s \| \Lw \|_{|s|,\Omega}^2=\real s \, a(\Lw,\Lw) 
        =\real  f(s\Lw) \le C |s|\| \Lw \|_{|s|,\Omega}\left\Vert \Lu^i \right\Vert_{H^{1}(D_1\setminus \overline{D})} \,.
    \end{align*}
    Rearranging yields 
    \begin{align}\label{est:w-hat}
        \| \Lw \|_{|s|,\Omega} \le C\dfrac{|s|}{\real s}\left\Vert \Lu^i \right\Vert_{H^{1}(D_1\setminus \overline{D})} \, .
    \end{align}
	The weak formulation of \eqref{eq:HHeq} with $D$ and $\Lu$ replaced by $D_\bfh$ and $\Lu_\bfh$, respectively, can be formulated as the task to find $\Lw_\bfh = \Lu_\bfh + \Lu^i\chi$ such that $a_\bfh(\Ltw_\bfh,v) = f(v)$ for all $v \in H_0^1(\Omega)$ where $\Ltw_\bfh = \Lw_\bfh \circ \varphi$ and
	\begin{align}\label{eq:ah}
	a_\bfh(\Ltw_\bfh,v) \,:=\, \int_{\Omega} \left(\nabla \Ltw_\bfh \cdot J_\varphi^{-1}(J_\varphi^{-1})^\top\overline{\nabla v} + s^2 \Ltw_\bfh \overline{v}\right) \text{det}(J_\varphi) \dx \, .
	\end{align}
The term $J_\varphi$ denotes the Jacobian of $\varphi$. The Riesz representation theorem shows that there is a well-defined boundedly invertible linear operator 
	$T: H_0^1(\Omega) \to H_0^1(\Omega)$ such that $\langle T w, v \rangle_{|s|,\Omega} = a(w,v)$ for all $w,v \in H_0^1(\Omega)$ with the scalar product defined in \eqref{eq:skpS}.
In the same way, there is a bounded linear operator $T_\bfh: H_0^1(\Omega) \to H_0^1(\Omega)$ such that 
$\langle T_\bfh w, v \rangle_{|s|,\Omega} = a_\bfh(w,v)$ for all $w,v \in H_0^1(\Omega)$.
	
By \cite[Lem.\@ 2.2]{Hett99} (see also \cite[Lem.\@ 3.2]{Hag19}) it holds that
	\begin{subequations}\label{eq:esth}
	\begin{align}
	\left\Vert \text{det}(J_\varphi) -1 -\div(\bfh) \right\Vert_{L^\infty(\Omega)} \, &\leq \, C \left\Vert \bfh \right\Vert_{C^1(\Gamma)}^2 \, , \\
	\left\Vert J_\varphi^{-1}(J_\varphi^{-1})^\top \text{det}(J_\varphi) - I +J_\bfh +  J_\bfh^\top - \div(\bfh) \right\Vert_{L^\infty(\Omega)} \, &\leq \, C \left\Vert \bfh \right\Vert_{C^1(\Gamma)}^2 \, .
	\end{align}
	\end{subequations}	
	Moreover, since $(J_\varphi^{-1})^\top= I - J_\bfh^{\top} + \mathcal{O}(\Vert \bfh \Vert_{C^1(\Gamma)}^2)$, we find for $\Vert \bfh \Vert_{C^1(\Gamma)} < h_0$ with $h_0>0$ sufficiently small that
	\begin{align*}
	c_1 |\bfx| \, \leq \, |(J_\varphi^{-1})^\top \bfx|  \, \leq c_2 |\bfx| \quad \text{and} \quad  
	 c_3 \, \leq \, \text{det}(J_\varphi) \, \leq \, c_4\, \quad \text{with } c_j>0\; \text{for } 1\, \leq\,  j \leq 4\, , \; \bfx \in \R^d\, .
	\end{align*}
	Thus, for any $w \in H_0^1(\Omega)$,
	\begin{align*}
	\min\{c_1^2 c_3, c_3 \}\real s \Vert w \Vert_{|s|,\Omega}^2 
	\, &=  \, \min\{c_1^2 c_3, c_3 \} \real s \int_\Omega |\nabla w|^2 + |s|^2 |w|^2 \dx \\
	\, &\leq \, \real s \int_\Omega \left( |J_\varphi^{-\top}\nabla w|^2 + |s|^2 |w|^2 \right)\text{det}(J_\varphi) \dx
	 \, = \, \real \left( a_\bfh(w,sw) \right)  \\
	 \, &\leq \, | a_\bfh(w,sw) | 
	 \, = \, |s| |\langle T_\bfh w, w \rangle_{|s|,\Omega}| \, \leq \, |s| \Vert T_\bfh w \Vert_{|s|,\Omega} \Vert w \Vert_{|s|,\Omega} \, ,
	\end{align*}
	implying that $T_\bfh^{-1}$ exists for $\Vert \bfh \Vert_{C^1(\Gamma)}<h_0$ and 
	\begin{align}\label{boundTh}
	\Vert T_\bfh^{-1} \Vert_{H_0^1(\Omega) \leftarrow H_0^1(\Omega)} \leq C \frac{|s|}{\real s}
	\end{align}
	 for a constant $C>0$ that does not depend on $\Vert \bfh \Vert_{C^1(\Gamma)}$ and $s$.
	Moreover,
	\begin{multline*}
	\Vert (T_\bfh - T) w \Vert_{|s|,\Omega}^2 \, = \, \left| a_\bfh(w, (T_\bfh - T) w) - a(w,(T_\bfh - T) w) \right| \\ 
	\, = \, 
	\left|\int_\Omega \nabla w \cdot \left( J_\varphi^{-1}(J_\varphi^{-1})^\top \text{det}(J_\varphi) - I \right) \nabla\big( \overline{\left( T_\bfh - T \right) w}\big) + s^2 w \left( \text{det}(J_\varphi) -1 \right) \overline{\left( T_\bfh - T \right) w} \dx \right| \, 
	\end{multline*}
	and by \eqref{eq:esth}, combined with the Cauchy--Schwarz inequality, this yields
	\begin{equation}\label{boundThmT}
	\Vert (T_\bfh - T) w \Vert_{|s|,\Omega} \, \leq \, C \Vert w \Vert_{|s|,\Omega}\left\Vert \bfh \right\Vert_{C^1(\Gamma)}\, .
	\end{equation}
	By the Riesz representation theorem, the weak formulations to find $\Lw,\widetilde{w}_\bfh$ such that $a(\Lw,v) = f(v)$ and $a_\bfh(\widetilde{w}_\bfh,v) = f(v)$ for all $v \in H_0^1(\Omega)$ are equivalent to determine $\Lw,\widetilde{w}_\bfh$ such that $T \Lw = F$ and $T_\bfh \widetilde{w}_\bfh = F$ for a $F \in H_0^1(\Omega)$.
	Thus,
	$T_\bfh (\widetilde{w}_\bfh-\Lw) \, = \, (T-T_\bfh) \Lw$ and using \eqref{boundTh}, \eqref{boundThmT} and \eqref{est:w-hat} implies that
	\begin{align}\label{eq:contih}
	\Vert \Ltw_\bfh - \Lw \Vert_{|s|,\Omega} \, \leq \, C \frac{|s|^{2}}{(\real s)^2} \left\Vert \Lu^i \right\Vert_{H^{1}(D_1\setminus \overline{D})} \Vert \bfh \Vert_{C^1(\Gamma)} \, .
	\end{align}
	Next, we show that there is a function $W \in H_0^1(\Omega)$ such that
	\begin{align}\label{eq:matder}
	\frac{1}{\Vert \bfh \Vert_{C^1(\Gamma)}} \left\Vert \Ltw_\bfh - \Lw - W \right\Vert_{|s|,\Omega} \, \leq \, C \frac{|s|^{3}}{(\real s)^{3}} \left\Vert \Lu^i \right\Vert_{H^{1}(D_1\setminus \overline{D} )}\Vert \bfh \Vert_{C^1(\Gamma)} \, .
	\end{align}
	Using $a(\Lw,v) = f(v)$ for all $v \in H_0^1(\Omega)$ with $a$ and $f$ from \eqref{eq:andf} together with \eqref{eq:ah} yields
	\begin{align}\label{eq:a1}
	a(\Ltw_\bfh - \Lw, v) \, = \, - \bigg( \int_\Omega \nabla \Ltw_\bfh \cdot \left( J_\varphi^{-1}(J_\varphi^{-1})^\top \text{det}(J_\varphi) - I \right) \overline{\nabla v} + s^2 \Ltw_\bfh ( \text{det}(J_\varphi) -1 ) \overline{v} \dx \bigg)
	\end{align}
	for all $v \in H_0^1(\Omega)$. We define $W \in H_0^1(\Omega)$ to be the unique solution of 
	\begin{align}\label{eq:a2}
	a(W,v) \, = \, \int_\Omega \nabla \Lw \cdot \left( J_\bfh + J_\bfh^\top - \div(\bfh)I \right) \overline{\nabla v} - s^2 \Lw \div(\bfh) \overline{v} \dx \, \quad \text{for all } v \in H_0^1(\Omega) \, .
	\end{align}
Then, by using \eqref{eq:a1}, \eqref{eq:a2}, \eqref{eq:esth}, the Cauchy--Schwarz inequality, \eqref{eq:contih} and \eqref{est:w-hat}, we find that 
\begin{align*}
&\frac{1}{\Vert \bfh \Vert_{C^1(\Gamma)}} \left| a\left( \Ltw_\bfh - \Lw - W, v \right)\right| \\ 
&\, \leq \, 
\frac{1}{\Vert \bfh \Vert_{C^1(\Gamma)}} \bigg| \int_\Omega \nabla (\Ltw_\bfh-\Lw) \cdot \left( J_\varphi^{-1}(J_\varphi^{-1})^\top \text{det}(J_\varphi) - I \right) \overline{\nabla v} + s^2 (\Ltw_\bfh-\Lw) ( \text{det}(J_\varphi) -1 ) \overline{v} \dx \\
&\phantom{\,\leq \, \frac{1}{\Vert \bfh \Vert_{C^1(\Gamma)}} \bigg|} + 
\int_\Omega \bigg( \nabla \Lw \cdot \left(J_\varphi^{-1}(J_\varphi^{-1})^\top \text{det}(J_\varphi) - I  +  J_\bfh + J_\bfh^\top - \div(\bfh)I \right) \overline{\nabla v} \\
&\phantom{\,\leq \,\frac{1}{\Vert \bfh \Vert_{C^1(\Gamma)}} \bigg|+ 
\int_\Omega }+ s^2 \Lw \left( \text{det}(J_\varphi) -1 - \div(\bfh) \right) \overline{v} \bigg) \dx \bigg| \\
&\, \leq \, C\left(\left\Vert  \Ltw_\bfh - \Lw \right\Vert_{|s|, \Omega} \Vert v \Vert_{|s|,\Omega} +  
\left\Vert  \Lw \right\Vert_{|s|, \Omega} \Vert v \Vert_{|s|,\Omega} \Vert \bfh \Vert_{C^1(\Gamma)}\right) \\
& \, \leq \, C\frac{|s|^{2}}{(\real s)^2} \left\Vert \Lu^i \right\Vert_{H^{1}(D_1\setminus \overline{D} )} \Vert v \Vert_{|s|,\Omega}\Vert \bfh \Vert_{C^1(\Gamma)}\, .
\end{align*}
Using $v = \Ltw_\bfh - \Lw - W$ and $\real s \Vert v \Vert_{|s|,\Omega} \leq |s| |a(v,v)|$ (see also \cite[Proof of Lem.\@ 4.9]{BanSay22}) shows \eqref{eq:matder}.

It is left to show that the function $W$ defined by \eqref{eq:a2} has the representation $W = \Lu' + \bfh \cdot \nabla \Lw$, where $\Lw = \Lu + \Lu^i\chi  \in H_0^1(\Omega)$ with $\Lu$ as the unique solution to \eqref{eq:HHeq} and $\Lu' \in H^1(\Omega)$ is the domain derivative defined by \eqref{eq:HHUp}.
Due to our assumption that $D$ is a $C^2$-domain, we have that $\Lw \in H^2(\Omega)$ (see Proposition~\ref{prop:H2-bound}).
For all $v \in H^2(\Omega) \cap H_0^1(\Omega)$ it holds that (see also \cite[p.\@ 87]{Kirsch93})
\begin{multline*}
\nabla \Lw \cdot \left( J_\bfh + J_\bfh^\top - \text{div}(\bfh)I \right) \nabla v  \\
\, = \, \text{div}\left( (\bfh \cdot \nabla v) \nabla \Lw + (\bfh \cdot \nabla \Lw) \nabla v - ( \nabla v \cdot \nabla \Lw)\bfh \right) - (\bfh \cdot \nabla \Lw ) \Delta v - (\bfh \cdot \nabla v ) \Delta \Lw
\end{multline*}
and the use of Green's formula gives that
\begin{multline}\label{eq:aWv}
a(W,v) \, = \, - \left(\int_\Omega  ( \bfh \cdot \nabla \Lw) \overline{\Delta v} + (\bfh \cdot \overline{\nabla v}) \Delta \Lw + s^2 \div(\bfh) \Lw \overline{v}  \dx\right) \\
+ \left(\int_{\Gamma} (\bfh\cdot \overline{\nabla v}) (\bfnu \cdot \nabla \Lw) + (\bfh\cdot \nabla \Lw) (\bfnu \cdot \overline{\nabla v}) - (\overline{\nabla v} \cdot \nabla \Lw) (\bfnu \cdot \bfh)  \ds(\bfx)\right)
\end{multline}
for all $v \in H^2(\Omega) \cap H_0^1(\Omega)$. For any $v \in H^2(\Omega)\cap H_0^1(\Omega)$ it holds that $\gamma(\nabla v) = (\bfnu \cdot \gamma(\nabla v))\bfnu$ and thus, using Green's formula in \eqref{eq:aWv} together with the fact that $( \bfh \cdot \overline{\nabla v})\Delta \Lw = s^2 ( \bfh \cdot \overline{\nabla v})\Lw$ thus give 
\begin{align*}
a(W,v) \, &= \, \int_\Omega \overline{\nabla v} \cdot \nabla( \bfh \cdot \nabla \Lw) - ( \bfh \cdot \overline{\nabla v}) \Delta \Lw - s^2 \Lw \overline{v} \div(\bfh) \dx \\
&= \, \int_\Omega \overline{\nabla v} \cdot \nabla( \bfh \cdot \nabla \Lw) - s^2\left( \div(\Lw \overline{v} \bfh) - (\bfh \cdot \nabla \Lw) \overline{v} \right) \dx \\
&= \, \int_\Omega \overline{\nabla v} \cdot \nabla( \bfh \cdot \nabla \Lw) + s^2(\bfh \cdot \nabla \Lw) \overline{v} \dx  \, = \, a( \bfh \cdot \nabla \Lw, v) 
\end{align*}
and thus, $a(W-\bfh \cdot \nabla \Lw ,v) = 0$ for all $v \in H^2(\Omega) \cap H_0^1(\Omega)$.
Since the domain derivative $\Lu'\in H^1(\Omega)$ defined in \eqref{eq:HHUp} also satisfies $a(\Lu' ,v) = 0$ for all $v \in H^2(\Omega) \cap H_0^1(\Omega)$ and additionally, $\gamma\left(W-\bfh \cdot \nabla \Lw\right) = -(\bfh \cdot \bfnu)\partial_\bfnu \Lw = \gamma \Lu'$, we conclude that $\Lu' = W - \bfh\cdot \nabla \Lw$.

Finally, let $\bfz \in \Omega$ and let $h_0>0$ be such that $D_\bfh \subset \subset D_0$ for all $\Vert \bfh \Vert_{C^1(\Gamma)} \leq h_0$ and $\bfz \notin D_0$. Then, using that $\Lu_\bfh(\bfz) - \Lu(\bfz) = \widetilde{w}_\bfh(\bfz) - \Lw(\bfz)$, the equality $\Lu'(\bfz) = W(\bfz)$, the representation formula in \eqref{eq:representation} with both integral operators $S(s)$ and $D(s)$ integrating over $\partial D_0$, the pointwise estimates in \eqref{eq:pwbound}, \eqref{eq:equiv} and the bound in \eqref{eq:matder}, we find that
\begin{align*}
&|\Lu_\bfh(\bfz) - \Lu(\bfz) - \Lu'(\bfz)|  \, = \, \left|-(S(s) \left( \partial_\bfnu \left( \widetilde{w}_\bfh - \Lw - W \right) \right))(\bfz) + (D(s) \left(\gamma\left(  \widetilde{w}_\bfh - \Lw - W \right)\right))(\bfz)  \right| \\
\, &\leq \, C(\sigma,\text{dist}(\bfz,\partial D_0)) |s|^{(d-1)/2} e^{-\mathrm{dist}(\bfz,\partial D_0)\real s} \left\Vert \widetilde{w}_\bfh - \Lw - W\right\Vert_{H^1(\R^d \setminus \overline{D_0})} \\
\, &\leq \, C(\sigma,\text{dist}(\bfz,\partial D_0)) |s|^{(d-1)/2}\max\{1,|s|^{-1} \}e^{-\mathrm{dist}(\bfz,\partial D_0)\real s} \left\Vert \widetilde{w}_\bfh - \Lw - W\right\Vert_{|s|,\Omega} \\
\, & \leq \, C(\sigma,\text{dist}(\bfz,\partial D_0)) e^{-\mathrm{dist}(\bfz,\partial D_0)\real s}  \frac{|s|^{(d+5)/2}}{(\real s)^3} \left\Vert \Lu^i \right\Vert_{H^{1}(D_1\setminus \overline{D} )} \Vert \bfh \Vert_{C^1(\Gamma)}^2 \, .
\end{align*}
This ends the proof.
\end{proof}
\begin{rem}
The regularity assumption on $D$ to be a $C^2$ domain together with $\bfh \in C^1(\Gamma,\R^d)$ are sufficient to guarantee the existence of the domain derivative in the Laplace domain, even though $D_\bfh = \varphi(D)$ with $\varphi$ from \eqref{eq:phi-diff} would not be $C^2$ anymore. 
In applications such as, for example, an iterative shape optimization, higher regularities, e.g. $\bfh \in C^2(\Gamma,\R^d)$, are required for the next iterate to be a $C^2$ domain and consequently for the next domain derivative to exist.
\end{rem}

\subsection{Frequency-explicit bounds for the domain derivative}
 
 Applying the bounds for the potential and boundary operators yields the following results.
 
 \begin{proposition}\label{prop:th-domain-derivative}
	Let $\real s \geq \sigma >0$, the boundary $\Gamma=\partial D$ at least $C^2$ and $\bfh\in C^1(\Gamma,\mathbb R^d)$. 
	Let further $\Lu'\in H^1(\Omega)$ denote the domain derivative from  \eqref{eq:HHUp}. 
	Then, the following bound holds in the natural $H^1$-norm
	\begin{align}\label{eq:esth1}
		\left\Vert \Lu' \right\Vert_{H^1(\Omega)} \, \le \,  C_\sigma \frac{|s|^4}{(\real s)^{3/2}} \left\Vert \gamma \Lu^i \right\Vert_{H^{3/2}(\Gamma)} \, .
	\end{align}
 Moreover, we have the following bound with respect to the $L^2$-norm 
  \begin{align}\label{eq:estL2}
		\left\Vert \Lu' \right\Vert_{L^2(\Omega)} \, \le \, C_\sigma \frac{|s|^3}{(\real s)^{3/2}} \left\Vert  \gamma\Lu^i \right\Vert_{H^{3/2}(\Gamma)} \, .
	\end{align}
	The constant $C_\sigma$ in those estimates depends only on the boundary $\Gamma$ and polynomially on $\sigma^{-1}$.
Finally, for any point $\bfz \in \Omega$ away from the boundary, we have the pointwise estimate
 	\begin{align}\label{eq:estpw}
 	\left| \Lu' (\bfz) \right|
	\, \le \,  C(\sigma,\mathrm{dist}(\bfz,\Gamma)) |s|^{(d+8)/2}	e^{-\mathrm{dist}(\bfz,\Gamma )\real s} \left\Vert \gamma\Lu^i \right\Vert_{H^{{3}/2}(\Gamma)} \, . 
\end{align}
The constant $C(\sigma,\mathrm{dist}(\bfz,\Gamma))$ depends polynomially on $\sigma^{-1}$ and on $\mathrm{dist}(\bfz,\Gamma)^{-(d-1)/2}$ for the dimensions $d=2,3$.
\end{proposition}

 \begin{proof}
We start by applying the first Green identity to $\Lu'\in H_\Delta^1(\Omega)$, which implies that for any $ v \in H^1(\Omega)$ the representation
\begin{align}\label{eq:G1}
	\int_{\Omega} \nabla \Lu' \cdot \overline{\nabla v} +  \Delta \Lu'  \overline{v}  \, \dx
	\,=\, -\int_\Gamma (\partial_\bfnu\Lu')( \overline{\gamma v})  \ds(\bfx)\,
\end{align}  
holds true.
In \eqref{eq:G1} we set $ v = s\Lu'$, insert the Helmholtz equation for $\Lu'$ as well as the boundary condition from \eqref{eq:HHUp} and find that
\begin{align*}
	\int_{\Omega} \overline{s} \left| \nabla \Lu'\right|^2 +  s  \left|s \Lu'\right|^2    \dx
	\,=\, 
	-\overline{s}\int_\Gamma (\partial_{\bfnu}\Lu')  (\overline{\gamma \Lu'})  \ds(\bfx)
	 		\,=\,
	\overline{s}\int_\Gamma (\partial_{\bfnu}\Lu')  \left(\bfh \cdot \bfnu\right)\overline{\partial_{\bfnu} (\Lu + \Lu^i)}  \ds(\bfx) \,.
\end{align*}  	
We continue by estimating the real part by the modulus from above to obtain
 	\begin{align}\label{prop:intermediate}
 		\real s	\int_{\Omega}  \left| \nabla \Lu'\right|^2 +    \left|s \Lu'\right|^2   \, \mathrm d \boldsymbol x
 		\,\le\, \left|s\right| \left\Vert \partial_{\bfnu}\Lu' \right\Vert_{H^{-1/2}(\Gamma)} \left\Vert \partial_{\bfnu} (\Lu+\Lu^i) \right\Vert_{H^{1/2}(\Gamma)}
 		\left\Vert \bfh \cdot \bfnu \right\Vert_{C^{1}(\Gamma)} \,.
 	\end{align}  
 	The factor $\left\Vert \bfh \cdot \bfnu \right\Vert_{C^{1}(\Gamma)}$ stays bounded for $\Gamma \in C^2$ and $\bfh\in C^1(\Gamma,\mathbb R^d)$. By \cite[Thm.\@ 4.16]{BanSay22} there is a positive constant $C_\sigma$, such that
 	\begin{align*}
 		\left\Vert  \partial_{\bfnu}\Lu' \right\Vert_{H^{-1/2}(\Gamma)}
 		&\, \le \, 
 		C_\sigma \frac{\left|s\right|^2}{\real s} \left\Vert  \gamma\Lu' \right\Vert_{H^{1/2}(\Gamma)}\, .
 	\end{align*}
 	Applying this estimate to the first factor of \eqref{prop:intermediate} and Proposition~\ref{prop:H2-bound} to the second one yields
 	\begin{align*}
 		\real s	\int_{\Omega}  \left| \nabla \Lu'\right|^2 +    \left|s \Lu'\right|^2  \dx	\le C_\sigma
 	\frac{\left|s\right|^{11/2}}{(\real s)^{3/2}} \left\Vert  \gamma\Lu' \right\Vert_{H^{1/2}(\Gamma)} \left\Vert \gamma \Lu^i \right\Vert_{H^{3/2}(\Gamma)}
 	\le C_\sigma
 	\frac{\left|s\right|^{8}}{{(\real s)^2}} \left\Vert  \gamma\Lu^i \right\Vert^2_{H^{3/2}(\Gamma)}\,,
 	\end{align*}
 	where we used Proposition~\ref{prop:H2-bound} again in the final inequality. 
 	Dividing by $\real s$ and taking the square root on both sides gives
 	\begin{align}\label{eq:esth1s}
 	\Vert \Lu' \Vert_{|s|,\Omega} \, \leq \, 
 	C_\sigma
 	\frac{\left|s\right|^{4}}{{(\real s)^{3/2}}} \left\Vert  \gamma\Lu^i \right\Vert_{H^{3/2}(\Gamma)}\, .
 	\end{align}
 	Using \eqref{eq:equiv} now shows \eqref{eq:esth1}. The estimate in terms of the $L^2$-norm in \eqref{eq:estL2} is obtained by omitting the first summand on the left-hand side and dividing through $|s|$ on both sides.
 	
 	We turn towards the stated estimate for point evaluations. Combining \cite[Lem.\@ 4.5]{BanSay22} with the estimate \eqref{eq:esth1s} gives
 	\begin{align*}
 	\left\Vert \partial_{\bfnu} \widehat u'\right\Vert^2_{H^{-1/2}(\Gamma)} 
 	&\,\le\,
 	|s|\int_{\Omega}  \left| \nabla \Lu'\right|^2 +    \left|s \Lu'\right|^2   \, \mathrm d \boldsymbol x 
 	\,\le \,
 	C_\sigma
 	{\left|s\right|^{9}} \left\Vert  \gamma\Lu^i \right\Vert^2_{H^{3/2}(\Gamma)} \,.
 	\end{align*}
 	By the trace theorem and \eqref{eq:esth1} we moreover obtain
 	\begin{align*}
 	 	\left\Vert  \gamma\Lu' \right\Vert_{H^{1/2}(\Gamma)}	
 	 	\,\le\,
 	 C_\Gamma	\left\Vert  \Lu' \right\Vert_{H^{1}(\Omega)}
 	 \,\le \,C_\sigma |s|^4 \left\Vert \gamma \Lu^i\right\Vert_{H^{3/2}(\Gamma)}\,.
 	\end{align*}
   Now we apply the representation formula in \eqref{eq:representation} and use the pointwise estimates in \eqref{eq:pwbound} to find that
   \begin{align*}
   \left| \Lu' (\bfz)\right|\,&=\, 	\left|- (\Sop(s) (\partial_{\bfnu} \Lu')) (\bfz) + (\Dop(s)(\gamma \Lu' ))(\bfz) \right| \\
   	\,&\le\, C(\sigma,\mathrm{dist}(\bfz,\Gamma)) e^{-\mathrm{dist}(\bfz,\Gamma )\real s}|s|^{(d-1)/2} \left(\left\Vert\partial_{\bfnu} \Lu' \right\Vert_{H^{-1/2}(\Gamma)} + \Vert \gamma \Lu' \Vert_{H^{1/2}(\Gamma)}\right)\,.
   \end{align*}
 The statement now follows by inserting the estimates for the traces of the scattered wave as before.
 This ends the proof.	
 \end{proof}

\subsection{Boundary integral equations for the domain derivative}
We can now find an integral formulation for the domain derivative in the Laplace domain. 
By \eqref{eq:SV-1-th} the solution to \eqref{eq:HHUp} is given by 
\begin{align*}
\Lu' \, = \,  -\Sop(s)\Vop^{-1}(s)\left((\bfh\cdot \bfnu)\partial_\bfnu(\Lu + \Lu^i)\right)\, .
\end{align*}
The aim is to replace $(\bfh\cdot \bfnu)\partial_\bfnu(\Lu + \Lu^i)$ by an operator taking a boundary density as an input. For this purpose, we define the linear and bounded operator $L(s) : H^{3/2}(\Gamma) \to H^{1/2}(\Gamma)$ by
	\begin{align}\label{eq:defLs}
	L(s)g \, := \,
	P_{\bfh\cdot  \bfnu}\left(\partial_{\bfnu,D}\Lambda_D - \partial_{\bfnu,\Omega}\Lambda_\Omega\right)g\,. 
	\end{align}
	In this definition, $P_{\bfh\cdot  \bfnu} : H^{1/2}(\Gamma)\to H^{1/2}(\Gamma)$ is given by 
		$P_{\bfh\cdot  \bfnu} \varphi =  -\left(\bfh \cdot \bfnu\right) \varphi$.	
Moreover, the operator $\partial_{\bfnu,\mathcal{O}}$ for $\mathcal{O} \in \{D,\Omega\}$ denotes the normal trace in $H^2(\mathcal{O})$ where $\bfnu$ is always directed into the domain $\Omega$. Furthermore, $\Lambda_{\mathcal{O}}:H^{3/2}(\Gamma) \to H^2(\mathcal{O})$ is defined by $\Lambda_{\mathcal{O}}\Lg = \Lv$, where $\Lv \in H^2(\mathcal{O})$ is the unique solution to \eqref{eq:HHeqg}.
The terms $\partial_{\bfnu,D}\Lambda_D$ and $\partial_{\bfnu,\Omega}\Lambda_\Omega$ are the interior and exterior Dirichlet-to-Neumann operators (see also \cite[Sec.\@ 4.8]{BanSay22}).
We define the linear and bounded operator $\mathcal{F}_\bfh:H^{3/2}(\Gamma) \to H^1(\Omega)$ by
\begin{align}\label{eq:F_hGen}
	\mathcal{F}_{\bfh}(s) \, :=\Sop(s)\Vop^{-1}(s)L(s) \, .
	\end{align}
The operator family $\mathcal{F}_\bfh(s)$ maps the incoming wave to the domain derivative in the Laplace domain, i.e., the Fr\'echet derivative of the map $\widehat{F}_{\Gamma}$ from \eqref{eq:domainder} at $0$ applied to $\bfh \in C^1(\Gamma,\R^d)$
may be written as
	\begin{align}\label{eq:F_h-ui}
	\widehat{F}_{\Gamma}'(0)\bfh \,=\,
		(\mathcal{F}_{\bfh}(s)\gamma	\Lu^i)(\bfz_j))_{j=1,\dots,M} \,.
	\end{align}
By Proposition~\ref{prop:th-domain-derivative} we obtain bounds for $\mathcal{F}_\bfh(s)\Lg$ with $\mathcal{F}_\bfh(s)$ from \eqref{eq:F_hGen} in different norms. Moreover, if $\Lg$ is the Dirichlet trace of an interior
Helmholtz solution, then the definition of $L(s)$ from \eqref{eq:defLs} admits a simpler, more useful form.
This is collected in the next corollary.
\begin{cor}
The operator family $\mathcal{F}_\bfh(s)$ from \eqref{eq:F_hGen} is bounded by
\begin{align*}
\Vert \mathcal{F}_\bfh(s) \Vert_{H^{1}(\Omega) \leftarrow H^{3/2}(\Gamma)} \, &\leq \, C_\sigma\frac{|s|^4}{(\real s)^{3/2}} \, ,\\
\Vert \mathcal{F}_\bfh(s) \Vert_{L^2(\Omega) \leftarrow H^{3/2}(\Gamma)} \, &\leq \, C_\sigma\frac{|s|^3}{(\real s)^{3/2}} \, , \\
\left| \mathcal{F}_\bfh(s)\cdot (\bfz) \right|_{\C \leftarrow H^{3/2}(\Gamma)} \, &\leq \, C(\sigma,\mathrm{dist}(\bfz,\Gamma)) |s|^{(d+8)/2}	e^{-\mathrm{dist}(\bfz,\Gamma )\real s} \, , \quad \bfz \in \Omega \, .
\end{align*}
The constants have the same properties as the respective constants in Proposition \ref{prop:th-domain-derivative}. 
Moreover, if $\Lg = \gamma \Lu^i|_D \in H^{3/2}(\Gamma)$ for $\Lu^i|_D \in H^2(D)$ satisfying 
$s^2\Lu^i - \Delta \Lu^i = 0$ in $D$, then 
\begin{align}\label{eq:FhHE}
	\mathcal{F}_{\bfh}(s)\gamma\Lu^i \, = \, \Sop(s)\Vop^{-1}(s)P_{\bfh\cdot  \bfnu}\Vop^{-1}(s)\gamma u^i \, .
	\end{align}
\end{cor}
\begin{proof}
The proof for the bounds can be done exactly as the proof of Proposition~\ref{prop:th-domain-derivative} by defining $\Lu' = \mathcal{F}_\bfh(s) \Lg$ for some $\Lg \in H^{3/2}(\Gamma)$ and proceeding as described in the proof.

For a function $\Lu^i$ as defined in the corollary, it holds that $\Lambda_D \gamma \Lu^i = \Lu^i$ by the uniqueness of the interior problem.
Therefore, by the definition of $L(s)$, the orientation of the unit normal $\bfnu$, \eqref{eq:HHeq} and \eqref{eq:obtNeumann}, it holds that
\begin{align*}
L(s)\gamma \Lu^i \, = \, P_{\bfh\cdot  \bfnu}\left(\partial_{\bfnu,D}\Lambda_D - \partial_{\bfnu,\Omega}\Lambda_\Omega\right)\gamma\Lu^i \, = \,
 P_{\bfh\cdot \bfnu}\partial_\bfnu(\Lu + \Lu^i) \, = \, P_{\bfh\cdot \bfnu}\Vop^{-1}(s) \gamma \Lu^i\, ,
\end{align*}
what shows the claimed representation for $\mathcal{F}_\bfh(s)\gamma \Lu^i$.
\end{proof}

\section{The domain derivative in the time domain}
In order to carry over the results from the frequency domain into the time domain, we use the setting of temporal Hilbert spaces, which we shortly introduce in the following. Here, we rely on the standard terminology surrounding spatio-temporal convolutions, as introduced in \cite{Bam86a} and then later in \cite{L94}. A recent introduction to the mathematical background of retarded boundary integral equations is \cite[Ch.\@ 2]{BanSay22}.
\subsection{Temporal convolutions and Hilbert spaces}
\label{subsec:Z}

Consider the analytic family of bounded linear operators  $\cqK(s):\cqX\to \cqY$, $\real s \geq \sigma>0$,
which are defined between Hilbert spaces $\cqX$ and $\cqY$.
Let $K$ be polynomially bounded, i.e.\@ let there exist a real $\kappa \in \R$  and $\nu\ge 0$, and for every $\sigma >0$ let there exist $M_\sigma <\infty$, such that
\begin{equation}\label{eq:pol_bound}
	\| \cqK(s) \|_{\cqY\leftarrow \cqX} \, \le \,  M_\sigma \frac{|s|^\kappa}{(\real s)^\nu}, \qquad \real s \, \geq \, \sigma \, >\,  0 \, .
\end{equation}
Any $K$ that fulfills this bound is the Laplace transform of a distribution of finite order of differentiation with support on the non-negative real half-line $t \ge 0$ (see also \cite[Cor.\@ 2.4]{BanSay22}). For a temporal function $g:[0,T]\to X$, which is sufficiently regular when extended by~$0$ on the negative real half-line, we define the Heaviside operational calculus notation by
\begin{equation} \label{Heaviside}
		K(\partial_t)g \, := \, \mathcal{L}^{-1}\{K\} * g \, .
	\end{equation}
The temporal convolution-type operator $K(\partial_t)$ defined in \eqref{Heaviside} acts on causal distributions with values in $X$. Applied to a temporal distribution $g$, the expression $K(\partial_t)g$ is a causal distribution with values in $Y$ and its Laplace transform is given by $\mathcal{L}\{K(\partial_t)g\}(s) = K(s)\mathcal{L}\{g\}(s)$.

The associativity of convolutions and the product rule of Laplace transforms moreover yields, for two families of operators $K(s)$ and $L(s)$ mapping into compatible spaces, the composition rule
	\begin{equation}\label{comp-rule}
		K(\partial_t)L(\partial_t)g \, =\, (KL)(\partial_t)g\, .
	\end{equation}
Let $X$ be a Hilbert space and let $H^r(\R,X)$ be the Sobolev space of order $r\in \R$ of $X$-valued functions on $\mathbb R$. Moreover, on finite intervals $(0, T )$,
	we write
\begin{align*}
H_0^r(0,T;X) \, := \, \{g|_{(0,T)} \,:\, g \in H^r(\R,X)\ \text{ with }\ g = 0 \ \text{ on }\ (-\infty,0)\} \, ,
\end{align*}
	where the
	subscript 0 in ${H_0^r}$ only refers to the left end-point of the interval.
	The norm on $H_0^r(0,T;X)$, which may be defined via a quotient norm, is equivalent to the norm $\Vert \partial_t^r \cdot \Vert_{L^2(0,T;X)}$.
	Moreover, by the Plancherel formula, this norm is equivalent to the Laplace domain interpretation
	\begin{align*}
	\left( \int_{\sigma + i\R} |s|^{2r}\left\Vert \mathcal{L}\{g\}(s) \right\Vert_X^2 \ds \right)^{1/2} \, \quad \text{for any } \sigma >0 \, .
	\end{align*}

	The temporal convolutional operator defined in \eqref{Heaviside} is also bounded by the Plancherel formula.
	We formulate this standard result here, for the convenience of the reader (see \cite[Lem.\@ 2.1]{L94}):
	Let $K(s)$ be bounded by \eqref{eq:pol_bound} in the half-plane $\text{Re }s > 0$. Then, $K(\partial_t)$ extends by density to a bounded linear operator from $H^{r+\kappa}_0(0,T;X)$ to $H^r_0(0,T;Y)$, which fulfills the bound
	\begin{equation*}
		\| K(\partial_t) \|_{ H^{r}_0(0,T;Y) \leftarrow H^{r+\kappa}_0(0,T;X)} \, \le \,  e M_{1/ T}
	\end{equation*}
	for arbitrary real $r$. The right-hand side follows from inserting $\sigma=1/T$ into the remaining factor of the Plancherel formula, which reads $e^{\sigma T} M_\sigma$. Pointwise estimates (in time) follow by using the continuous embedding $H^{k+\alpha}_0(0,T;X)\subset C^k([0,T];X)$, which holds for any integer $k\ge 0$ and $\alpha>1/2$.
With these notations we can define (generalized) solutions to the wave equation \eqref{eq:scatwave-1}-\eqref{eq:scatwave-2}. We formulate this in the next corollary, which is a consequence of \eqref{eq:wpin}, \eqref{eq:equiv} as well as \eqref{bound-V-m-1}, \eqref{eq:pwbound1} and Proposition~\ref{prop:H2-bound} (see also \cite[Prop.\@ 4.11]{BanSay22}).
\begin{cor}\label{cor:wp}
The unique solution of \eqref{eq:scatwave-1}-\eqref{eq:scatwave-2} can be explicitly written by using the
frequency domain identity \eqref{eq:SV-1-th}, the Heaviside notation \eqref{Heaviside} and the composition rule \eqref{comp-rule} as
\begin{align}\label{eq:SV-1-td}
u\,=\, -(\Sop\Vop^{-1})(\partial_t) \gamma u^i\, .
\end{align}
 For $r \in \R$ the continuous convolution-type operator $(SV^{-1})(\partial_t)$ has the mapping properties
\begin{subequations}
\begin{align}
&(\Sop\Vop^{-1})(\partial_t) : H_0^{r+3/2}(0,T; H^{1/2}(\Gamma)) \to  H_0^r(0,T; H^{1}(\Omega))\, , \label{eq:map1} \\
&(\Sop\Vop^{-1})(\partial_t)\cdot (\bfz) : H_0^{r+(d+3)/2}(0,T; H^{1/2}(\Gamma)) \to  H_0^r(0,T; \R) \, , \quad \bfz \in \Omega\, , \label{eq:map2} \\
&(\Sop\Vop^{-1})(\partial_t) : H_0^{r+5/2}(0,T; {H}^{3/2}(\Gamma)) \to  H_0^r(0,T; H^{2}(\Omega)) \label{eq:map3}\, .
\end{align}
\end{subequations}
\end{cor}

\subsection{From frequency to time domain: The temporal domain derivative}
We are now in the position to define the temporal domain derivative. To be consistent with the frequency domain setting, we use a similar terminology as the one introduced in Section~\ref{sec:th-dom-deriv}. 
We consider the boundary-to-solution operator in the time-domain $X : \Gamma \mapsto (u(\bfz_j))_{j=1,\dots,M}$, where $u$ solves \eqref{eq:scatwave-1} together with \eqref{eq:scatwave-2}. Then, let $D(F_\Gamma)$ denote a neighborhood of the zero function in $C^1(\Gamma,\R^d)$ that is so small that
\begin{align}\label{eq:FGamma}
	F_\Gamma : D(F_\Gamma) \subset C^1(\Gamma,\R^d) \to (L^2(0,T))^M\, , \qquad F_\Gamma(\bfh) \, := \, X(\Gamma_\bfh)
\end{align}
is well-defined, where $\Gamma_\bfh$ is defined as in Section~\ref{sec:th-dom-deriv}.
We denote by $F_\Gamma'(0) :  C^1(\Gamma,\R^d) \to (L^2(0,T))^M$ the Fr\'echet derivative of $F_\Gamma$ at zero satisfying
\begin{align*}
	\frac{1}{\left\Vert \bfh \right\Vert_{C^1(\Gamma)}}\left\Vert F_\Gamma(\bfh) - F_\Gamma(0) - F_\Gamma'(0)\bfh \right\Vert_{(L^2(0,T))^M} \, \to \, 0\, , \quad \text{as } \left\Vert \bfh \right\Vert_{C^1(\Gamma)} \to 0 \, .
\end{align*}
As in the Laplace domain, this Fr\'echet derivative may be characterized using the temporal domain derivative $u'$, which is the solution to a time-dependent scattering problem. In the following, we discuss this scattering problem and the implications for the regularity of the time-dependent domain derivative.
\begin{proposition}\label{prop:td-dom-deriv}
	Let $d\in\{2,3\}$, $D$ be a bounded $C^2$ domain, $\bfh\in C^1(\Gamma,\mathbb R^d)$ and for $r\geq 0$ let $u^i\in H_0^{r+(d+8)/2}(0,T; H_{\mathrm{loc}}^{2}(\R^d))$.
	Moreover, let $u\in H_0^{r+(d+3)/2}(0,T; H^2(\Omega))$ be the unique solution of \eqref{eq:scatwave-1}--\eqref{eq:scatwave-2}. Then, the Fr\'echet derivative of $F_\Gamma$ at zero exists and $F_\Gamma'(0)\bfh$ is given by the solution $u'\in H_0^{r+d}(0,T;H^1(\Omega))$ of
	\begin{align}\label{eq:td_dd}
		\begin{split}
			\partial_t^2 u' - \Delta u'   \,=\, 0  \quad &\text{in } \Omega \times [0,T]\, , \\
			u' \, = \, -\left(\bfh \cdot \bfnu\right)\partial_\bfnu(u+u^i) \quad &\text{on } \Gamma \times [0,T] \, ,
		\end{split}
	\end{align}
	evaluated at $\bfz_j$, $j=1,\dots,M$, i.e.\@ $F_\Gamma'(0) \bfh = (u'(\bfz_j))_{j=1,\dots,M}$.
	The function $u'$ is called the temporal domain derivative.
\end{proposition}
\begin{proof}
 By the trace theorem, we obtain that $\gamma u^i\in H_0^{r+(d+8)/2}(0,T; H^{3/2}(\Gamma))$. The solution $u$ of \eqref{eq:scatwave-1}-\eqref{eq:scatwave-2} is, by \eqref{eq:map3},  in $H_0^{r+(d+3)/2}(0,T; H^{2}(\Omega))$. Thus, 
$\left(\bfh \cdot \bfnu\right)\partial_\bfnu(u+u^i) \in H_0^{r+(d+3)/2}(0,T; H^{1/2}(\Gamma))$ and \eqref{eq:map1} yields that $u' \in H_0^{r+d}(0,T; H^1(\Omega))$. 
Moreover, to study pointwise evaluations in $\Omega$, we use \eqref{eq:map2} to see that $u'(\bfz_j)\in H_0^r(0,T;\R)$ for $j=1,\dots,M$.

We apply the Plancherel formula and obtain, for any $\sigma>0$, that
\begin{multline*}
	\left\Vert F_\Gamma(\bfh) - F_\Gamma(0) - (u'(\bfz_j))_{j=1,\dots,M}\right\Vert_{(H^r(0,T))^M}^2
	 \\  \,\leq\,
	\rme^{2\sigma T}
\int_{\sigma +i\mathbb R}|s|^{2r}\left|\widehat{F}_\Gamma(\bfh) -\widehat{F}_\Gamma(0) - \widehat{F}_\Gamma'(0)\bfh \right|^2\ds\, ,
\end{multline*}
where we used that $(\widehat{u}'(\bfz_j))_{j=1,\dots,M} = \widehat{F}_\Gamma'(0)\bfh$, which holds by Proposition~\ref{prop:th-dom-deriv-bounds}, for $\Vert \bfh \Vert_{C^1(\Gamma)}<h_0$. 
We choose $\sigma=1/T$, use the bound \eqref{eq:ddest} and apply the Plancherel formula once more, to obtain
\begin{multline*}
	\frac{\rme^{2\sigma T}}{\left\Vert \bfh \right\Vert^2_{C^1(\Gamma)}}
\int_{\sigma +i\mathbb R}|s|^{2r}\left|\widehat{F}_\Gamma(\bfh) -\widehat{F}_\Gamma(0) - \widehat{F}_\Gamma'(0)\bfh \right|^2\ds \\
\, \leq \, C_T
	\Vert \bfh \Vert_{C^1(\Gamma)}^2	\int_{\sigma +i\mathbb R}  \left\Vert s^{(d+5)/2+r} \Lu^i \right\Vert_{H^{1}(D_1\setminus \overline{D})}^2  \ds
	 \, \leq \,  C_T
	\Vert \bfh \Vert_{C^1(\Gamma)}^2 \left\Vert u^i \right\Vert^2_{H_0^{(d+5)/2+r}(0,T; H^{1}(D_1\setminus \overline{D})) }.
\end{multline*}
The constant $C_T$ depends only polynomially on the final time $T$.
Taking the limit $\Vert\bfh\Vert_{C^1(\Gamma)}\rightarrow 0$ on the right-hand side shows the desired property.
\end{proof}
The frequency dependent bounds of Proposition~\ref{prop:th-domain-derivative} from the previous section are directly transferred to the time-domain
by using the same techniques that we applied in the proof of Proposition~\ref{prop:td-dom-deriv}.
 \begin{thm}\label{th:td-dd-bound}
 Let $D$ be a bounded domain, for which the boundary $\Gamma=\partial D$ is at least $C^2$ and $\bfh\in C^1(\Gamma,\mathbb R^d)$. Let $u'$ denote the solution to  \eqref{eq:td_dd}. Then, the following bound holds in the natural $H^1$-norm for any $r\geq 0$
	\begin{align*}
		\left\Vert u' \right\Vert_{H^r(0,T; H^1(\Omega))} \, \le \,  C_T \left\Vert  \gamma u^i \right\Vert_{H^{r+4}(0,T;H^{3/2}(\Gamma) )} \, .
	\end{align*}
	Moreover, we have the following bound with respect to the $L^2$-norm 
	\begin{align*}
		\left\Vert u' \right\Vert_{H^r(0,T; L^2(\Omega))} \, \le \, C_T \left\Vert \gamma u^i \right\Vert_{H^{r+3}(0,T;H^{3/2}(\Gamma) )} \, .
	\end{align*}
	The constant $C_T$ in those estimates depends only on the boundary $\Gamma$ and polynomially on the final time $T$.
	Finally, for any point $\bfz \in \Omega$ away from the boundary, we have the estimate
	\begin{align*}
		\left\Vert u' (\bfz) \right\Vert_{H^r(0,T)}
		\, \le \,  C_T	 \left\Vert \gamma u^i \right\Vert_{H^{r+(d+8)/2}(0,T;H^{3/2}(\Gamma) )} \, 
	\end{align*}
	for $d \in \{2,3\}$. The constant $C_T$ depends only on the boundary $\Gamma$, the distance of $\bfz$ to the boundary and polynomially on the final time $T$.
\end{thm}
Finally, by using the operator $\Fcal_\bfh$ from \eqref{eq:F_h-ui} and the representation formula in \eqref{eq:FhHE} we can write the domain derivative $u'$ from \eqref{eq:td_dd} as a convolution-type operator. 

\begin{cor}
For $\gamma u^i \in H_0^r(0,T; {H}^{3/2}(\Gamma))$
the temporal domain derivative $u'$ from \eqref{eq:td_dd} can be explicitly written by 
using the
frequency domain identities \eqref{eq:F_h-ui}, \eqref{eq:FhHE}, the Heaviside notation \eqref{Heaviside} and the composition rule \eqref{comp-rule} as
\begin{align}\label{eq:F_h-t}
u' \,=\,
	\mathcal{F}_{\bfh}(\partial_t)\gamma	u^i \, = \, 
	(SV^{-1}P_{\bfh\cdot\bfnu}V^{-1})(\partial_t) \gamma u^i
	\, .
\end{align}
For $r \geq 0$ the continuous convolution-type operator $\mathcal{F}_{\bfh}(\partial_t)$ has the mapping properties
\begin{align*}
&\mathcal{F}_{\bfh}(\partial_t) : H_0^{r+4}(0,T; {H}^{3/2}(\Gamma))) \to  H_0^r(0,T; H^{1}(\Omega))\, ,  \\
&\mathcal{F}_{\bfh}(\partial_t): H_0^{r+3}(0,T; {H}^{3/2}(\Gamma))) \to  H_0^r(0,T; L^2(\Omega)) \, \, , \\
&\mathcal{F}_{\bfh}(\partial_t)\cdot (\bfz) : H_0^{r+(d+8)/2}(0,T; {H}^{3/2}(\Gamma))) \to  H_0^r(0,T; \R) \,, \quad \bfz \in \Omega\,   .
\end{align*}
\end{cor}

\section{Semi-discretization in time by Runge--Kutta CQ}

\subsection{Recap: Runge--Kutta convolution quadrature}

We give a brief introduction on the approximation of temporal convolutions $K(\partial_t)g$ by the convolution quadrature method based on Runge--Kutta multistage methods. Consider an $m$-stage implicit Runge--Kutta  discretization of the initial value problem $y' = f(t,y)$, $y(0) = y_0$. With the constant time step size $\tau>0$, we aim to compute approximations $y^n$ to $y(t_n)$ at equidistant time points $t_n = n\tau$. Simultaneously, the method computes approximations at the internal stages $Y^{ni}$ approximating $y(t_n + \boldsymbol{c}_i \tau)$, by solving the system
\begin{equation*}
	\begin{aligned}
		Y^{ni} &\,=\, y^n + \tau \sum_{\ell = 1}^m a_{i\ell} f(t_n+c_\ell\tau,Y^{n\ell}), \qquad i =
		1,\dotsc,m \, ,\\
		y^{n+1} & \,=\, y^n + \tau \sum_{\ell = 1}^m b_\ell f(t_n+c_\ell\tau,Y^{n\ell}) \, .
	\end{aligned}
\end{equation*}
Details on Runge--Kutta methods can be found e.g. in \cite{HairerWannerII}. 
The scheme is fully determined by its coefficients, which are denoted by
\begin{equation*}
	\mathscr{A} = (a_{ij})_{i,j = 1}^m \,, \quad b = (b_1,\dotsc,b_m)^T\, ,
	\quad \text{and} \quad c = (c_1,\dotsc,c_m)^T \,.
\end{equation*}
The function $R(z) = 1 + z b^T ( I - z \mathscr{A})^{-1} \mathbbm{1}$, where $\mathbbm{1} = (1,1,\dotsc,1)^T \in \R^m$, is referred to as the stability function of the Runge--Kutta method. We always assume that $\mathscr{A}$ is invertible.

Convolution quadrature methods can be constructed with Runge--Kutta methods and are, in many settings, more efficient than their BDF-based counterparts (see e.g. \cite{B10,BLM11}).
An excellent book providing an overview of key results in the field and recent developments is \cite{BanSay22}.

Let $\cqK(s):\cqX\to \cqY$, $\real s \geq \sigma>0$ be an analytic family of bounded linear operators between Hilbert spaces $\cqX$ and $\cqY$, that satisfies \eqref{eq:pol_bound}.
By the results described in Section~\ref{subsec:Z}, this yields a temporal convolution operator $\cqK(\partial_t):H^{r+\kappa}_0(0,T;\cqX) \to \cqH^{r}_0(0,T;\cqY)$ for arbitrary real $r$. Consider now a time-dependent function $\cqg:[0,T]\to \cqX$ that is, together with its extension by 0 to the negative real half-axis $t<0$, sufficiently regular for the expression \eqref{Heaviside} to be well-defined. 
We approximate the convolution $(\cqK(\partial_t)\cqg)(t)$ at the discrete times  
$$
\vtn = (t_n+c_\ell \tau)_{\ell=1}^m\, ,\text{ where } t_n=n\tau\, ,
$$
i.e., at the equidistant time points, which are the stages of the underlying Runge--Kutta method. 

The Runge--Kutta differentiation symbol reads
\begin{equation*}
	\Delta(\zeta) \, := \, \Bigl(\mathscr{A}+\frac\zeta{1-\zeta}\mathbbm{1} \cqb^T\Bigr)^{-1} \in \C^{m \times m}\, , \qquad
	\zeta\in\C \hbox{ with } |\zeta|<1\, .
\end{equation*}
As a consequence of the Sherman--Woodbury formula, this expression is well defined for $|\zeta|<1$ if $R(\infty)=1-b^T\mathscr{A}^{-1}\bone$ satisfies $|R(\infty)|\le 1$. For A-stable Runge--Kutta methods (e.g.\@ the Radau IIA methods), the eigenvalues of the matrices $\Delta(\zeta)$ have positive real part for $|\zeta|<1$ (see \cite[Lem.\@ 3]{BLM11}).
The Sherman--Morrison formula then yields the expression
\begin{equation*}
	\Delta(\zeta) \,=\, \mathscr{A}^{-1} -\frac{\zeta}{1-R(\infty)\zeta} \mathscr{A}^{-1} \bone b^T \mathscr{A}^{-1}\, .
\end{equation*}
We are now in a position to define the convolution quadrature weights ${\cqW}_n(\cqK):\cqX^m \to \cqY^m$.

We replace the complex argument $s$ in $\cqK(s)$ by the matrix-valued analytic function $\Delta(\zeta)/\tau$ and write down the power series expansion
\begin{equation*}
	\cqK\Bigl(\frac{\Delta(\zeta)}\tau \Bigr) \,=\, \sum_{n=0}^\infty {\boldsymb W}_n(\cqK) \zeta^n\,.
\end{equation*}
 In the following, we use an upper index to denote a sequence element with $m$ components.
Thus, for a sequence $\cqg=(\cqg^n)$ with $\cqg^n=(\cqg^n_\ell)_{\ell=1}^m\in \cqX^m$ we arrive at the discrete convolution denoted by
\begin{equation}\label{rkcq}
	\bigl(\cqK(\underline{\partial_t^\tau}) \boldsymb g \bigr)^n \, := \, \sum_{j=0}^n {\boldsymb W}_{n-j}(\cqK) \boldsymb g^j \in \cqY^m \,.
\end{equation}
The notation $\cqK(\underline{\partial_t^\tau}) \boldsymb g$ in \eqref{rkcq} indicates that the resulting vector contains approximations at the stages $\vtn$.
For functions $\boldsymb g:[0,T]\to \cqX$, we use this notation for the vectors $\cqg^n = \cqg(\vtn) = \bigl(\cqg(t_n+c_i\tau)\bigr)_{i=1}^m$ of values of $\cqg$.
The $\ell$-th component of the vector $\bigl(\cqK(\underline{\partial_t^\tau}) \boldsymb g \bigr)^n$, that we denote by $\bigl(\cqK(\underline{\partial_t^\tau}) \boldsymb g \bigr)^{n,\ell}$, is then an approximation to $\bigl(\cqK(\partial_t)\cqg\bigr)(t_n+c_\ell\tau)$, i.e.\@ $\bigl(\cqK( \underline{\partial_t^\tau}) \boldsymb g \bigr)^{n,\ell} \approx \bigl(\cqK(\partial_t) \cqg \bigr)(t_{n}+c_\ell \tau)$
 (see \cite[Thm.\@ 4.2]{BL19}). 
In particular, if $c_m = 1$, as is the case with stiffly stable Runge--Kutta methods, which includes the Radau IIA methods,
the continuous convolution at $t_{n}$ is approximated by the $m$-th, i.e.~last component of the $m$-vector \eqref{rkcq} for $n-1$:
\begin{equation*}
	 \left(\cqK( \partial_t^\tau) \boldsymb g\right)_n \, := \,  \bigl(\cqK( \underline{\partial_t^\tau}) \boldsymb g \bigr)^{n-1,m} \in \cqY \, . 
\end{equation*}
These components approximate the continuous convolution at the equidistant time points $t_n$, i.e. $ \left(\cqK( \partial_t^\tau) \boldsymb g\right)_n\approx\bigl(\cqK(\partial_t) \cqg \bigr)(t_{n}) $.

A property that is key to show stability in many settings is that the composition rule~\eqref{comp-rule} is preserved under this discretization: For two compatible operator families $\cqK(s)$ and $\cqL(s)$, we have
\begin{equation*}
	\cqK(\underline{\partial_t^\tau})\cqL(\underline{\partial_t^\tau})\cqg \, =\,  (\cqK\cqL)(\underline{\partial_t^\tau})\cqg \, .
\end{equation*}
Such a property can only be formulated for the vector valued discrete convolution, which includes the approximations at the stages $\vtn$, but can not be formulated for the approximation $\cqK( \partial_t^\tau) \boldsymb g $ at the equidistant time points $t_n$.

The following error bound for Runge--Kutta convolution quadrature from~\cite[Thm.\@ 3]{BLM11}, here directly stated for the Radau IIA methods \cite[Sec.~IV.5]{HairerWannerII} and transferred to a Hilbert space setting, will be the basis for our error bounds of the  time discretization.

\begin{lem}
	\label{lem:RK-CQ}
	Let $\cqK(s):\cqX\to \cqY$, $\real s \geq \sigma>0$,  be an analytic family of linear operators between Banach spaces $\cqX$ and $\cqY$ satisfying
	the bound \eqref{eq:pol_bound} with exponents $\kappa$ and $\nu$.
	Consider the Runge--Kutta convolution quadrature based on the Radau IIA method with $m$ stages. Let $r>\max(2m+\kappa,2m-1,m+1)$ and further let $\cqg \in \cqC^r([0,T],\cqX)$ satisfy $\cqg(0)=\cqg'(0)=...=\cqg^{(r-1)}(0)=0$. Then, the following error bound holds at $t_n=n\tau\in[0,T]$:
	\begin{multline*}
	\left\|  \bigl(\cqK( \partial_t^\tau) \boldsymb g \bigr)_{n}-(\cqK(\partial_t)\cqg)(t_{n}) \right\|_{\cqY}
		\\ \, \le \, 
		C\, M_{1/T}\,\tau^{\min(2m-1,m+1-\kappa+\nu)}
		\left(\|{\cqg^{(r)}(0)}\|_{\cqX}+\int_0^t\|{\cqg^{(r+1)}(t')}\|_{\cqX} \,\mathrm{d}t'
		\right) \, .
	\end{multline*}
	The constant C is independent of $\tau$ and $\cqg$ and $M_\sigma$ of \eqref{eq:pol_bound}, but depends on the exponents $\kappa$ and $\nu$ in \eqref{eq:pol_bound} and on the final time $T$.
\end{lem}
In the next remark we comment on the selection of the Radau IIA method for the convolution quadrature.
\begin{rem}
Runge--Kutta convolution quadrature methods such as those based on the Radau IIA methods (see \cite[Sec.~IV.5]{HairerWannerII}), often enjoy more favourable properties than their multistep-based counterparts, which cannot exceed order~2. Recently, Runge--Kutta convolution quadrature approximations based on Gau\ss~methods have been analyzed in \cite{BF22}. 
However, note that the domain derivative in the frequency domain at a point $\boldsymbol z\in \Omega$ decays exponentially fast with respect to $\real s$ (see Equation~\eqref{eq:estpw}). As a consequence, Lemma \ref{lem:RK-CQ} implies convergence rates at the full classical order, which makes the Radau IIA based convolution quadrature schemes the ideal candidates for the present topic. Other stiffly accurate A-stable methods such as the Lobatto IIIC method would fulfill similar properties, but offer no benefit in return for their lower classical order $2m-2$.
\end{rem}
\subsection{The time-discrete domain derivative}
We apply convolution quadrature to the present scattering problems, starting with the initial value problem \eqref{eq:scatwave-1}--\eqref{eq:scatwave-2}. Discretizing the temporal convolution \eqref{eq:SV-1-td} yields the approximation
\begin{align*}
	u_\tau\,=\, -\Sop(\partial^\tau_t)\Vop^{-1}(\underline{\partial^\tau_t}) u^i\, .
\end{align*}
We use a similar notation as in \eqref{eq:FGamma} and define $X_\tau : \Gamma \mapsto (u_\tau(\bfz_j))_{j=1,\dots,M}$ and
\begin{align}\label{eq:defF}
	F_{\Gamma,\tau} : D(F_\Gamma)\subset C^1(\Gamma,\R^d) \to (\ell^2_\tau(0,N))^M \, , \qquad F_{\Gamma,\tau}(\bfh) \, := \, X_\tau(\Gamma_h)\, .
\end{align}
In particular, we note that $F_{\Gamma,\tau}(0)$ is the sequence ($u_\tau(\bfz_j))_{j=1,\dots,M}$, whose elements approximate the scattered wave $u$ from \eqref{eq:scatwave-1}--\eqref{eq:scatwave-2} evaluated at the spatial points $\bfz_j$, $j=1,\dots,M$ and at the equidistant time points $t_n$ for $n=1,\dots,N$ (which are the $m$-th components of the stages $\underline{t_n}$ for $n=0,\dots,N-1$).
Here, we use the notation $\ell^2_\tau (0,N)=\mathbb R^{N}$ for the space of finite sequences endowed with the norm (see \cite{BL19})
\begin{align}\label{eq:seq-norm}
	\|u_\tau (\bfz)\|^2_{\ell^2_\tau(0,N)}
	\, := \, \tau\sum_{n=1}^N |(u_\tau)_n(\bfz)|^2 \, .
\end{align}
This norm is a time-discrete pendant to the $L^2(0,T)$-norm, which scales appropriately with the finite time $T$. More generally, we write $\ell^2_\tau(0,N;V)$ for an arbitrary Hilbert space $V$, where the absolute value $|\cdot|$ in \eqref{eq:seq-norm} is exchanged by the appropriate norm $\|\cdot\|_V$.
Note that by applying Lemma~\ref{lem:RK-CQ} to the estimates \eqref{bound-V-m-1} and \eqref{eq:pwbound}, we obtain the pointwise error estimate for the scattering problem 
\begin{align}\label{eq:pwboundstd}
| F_{\Gamma,\tau}(0)_n-F_\Gamma(0)(t_n) |
\, \le \, C\tau^{2m-1} \left\Vert \gamma u^i\right\Vert_{H_0^{\widetilde{r}}(0,T; H^{1/2}(\Gamma))}\, \quad \text{for } \widetilde{r}\, >\, 2m+3+\frac{d}{2}
\end{align}
(see also \cite[Thm.\@ 4]{BLM11}).
The requirement on the regularity of $\gamma u^i$ on the right hand side of \eqref{eq:pwboundstd} can be seen by using $\kappa = (3+d)/2$, arbitrary $\nu>0$ and $r\in \N$ with $r  >  2m+\kappa$ in Lemma \ref{lem:RK-CQ} and by noting that $H_0^{\widetilde{r}}(0,T;X) \subset C^{r+1}([0,T],X)$ for $\widetilde{r}>r+3/2$. 
Lower order error bounds can be derived under lower regularity assumptions on the incident wave $\gamma u^i$.
Again, we denote by $F_{\Gamma,\tau}'(0) : C^1(\Gamma,\R^d) \to \ell^2(0,N;\R^M)$ the Fr\'echet derivative of the operator $F_{\Gamma,\tau}$ at zero, i.e.
\begin{align*}
	\frac{1}{\left\Vert \bfh \right\Vert_{C^1(\Gamma)}}\left\Vert F_{\Gamma,\tau}(\bfh) - F_{\Gamma,\tau}(0) - F_{\Gamma,\tau}'(0)\bfh \right\Vert_{\ell^2(0,N;\R^M)} \, \to \, 0\, , \quad \text{as } \left\Vert \bfh \right\Vert_{C^1(\Gamma)} \to 0\, .
\end{align*}
The time-discrete domain derivative is characterized by the convolution quadrature discretization of the scattering problem, which is formulated in the following corollary.
\begin{cor}
Let $\left( \gamma u^i(\vtn) \right)_{n\ge 0}\in \ell^2_{\tau}(0,N ; H^{3/2}(\Gamma)^m).$ 
Then, the Fr\'echet derivative of the time-discrete scattering operator $F_{\Gamma,\tau}$ at zero from \eqref{eq:defF} is given by the convolution quadrature of the temporal convolution \eqref{eq:F_h-t} evaluated at $\bfz_j$, $j=1,\dots,M$, that is
\begin{align}\label{eq:F_h-t-tau}
(u_\tau'(\bfz_j))_{j=1,\dots,M} = 	((\mathcal{F}_{\bfh}(\partial^\tau_t)\gamma u^i)(\bfz_j))_{j=1,\dots,M}  \, = \, F_{\Gamma,\tau}'(0)\bfh \,.
\end{align}
\end{cor} 
\begin{proof}
The statement is the direct consequence of applying the frequency domain counterpart, i.e., Proposition~\ref{prop:th-dom-deriv-bounds} to the generating function of $u'_\tau$.
\end{proof}
Analogous to the frequency domain and the time-continuous settings, we can formulate the time-discrete domain derivative $F_{\Gamma,\tau}'(0)\bfh$ as the evaluation of the solution $u_\tau'$ of a time-discrete scattering problem.
Finally, applying the general approximation result Lemma~\ref{lem:RK-CQ} with the Laplace domain bound of Proposition~\ref{prop:th-domain-derivative}, yields the following error estimate for the time-discrete domain derivative.

\begin{thm}\label{th:cq-conv}
Let $d\in\{2,3\}$, and $\gamma u^i\in  H_0^{\widetilde{r}}(0,T;H^{3/2}(\Gamma))$ with $\widetilde{r} >2m+(d+11)/2$. Consider the convolution quadrature semi-discretization \eqref{eq:F_h-t-tau} of the scattering problem equivalent to the temporal domain derivative, based on the $m$-stage Radau IIA method. Then, for any point $\bfz \in \Omega$ away from the boundary, we have the following estimate at $t_n = n\tau$:
\begin{align*}
	\left|(u_\tau')_n (\bfz) -u'(\bfz,t_n) \right|
	\, \le \,  C	\tau^{2m-1} \left\Vert \gamma u^i \right\Vert_{ H_0^{\widetilde{r}}(0,T;H^{3/2}(\Gamma))} .
\end{align*}
The constant $C$ depends polynomially on the final time $T$, the domain $\Gamma$ and inversely polynomial on the distance of $\bfz$ from the boundary $\Gamma$.
\end{thm}

\begin{rem}
Applying the other bounds of Proposition~\ref{prop:th-domain-derivative} yields the estimates
	\begin{align*}
	\left\Vert (u_\tau')_n -u'(\cdot,t_n)\right\Vert_{H^1(\Omega)}	 \, &\leq \, C	\tau^{m-3/2} \left\Vert \gamma u^i \right\Vert_{ H_0^{2m+6}(0,T;H^{3/2}(\Gamma))} \,  ,\\
	\left\Vert (u_\tau')_n -u'(\cdot,t_n)\right\Vert_{L^2(\Omega)}	 \, &\leq \, C	\tau^{m-1/2} \left\Vert \gamma u^i \right\Vert_{ H_0^{2m+5}(0,T;H^{3/2}(\Gamma))}.
\end{align*}
\end{rem}
\begin{rem} The above assumptions on the temporal regularity of the incoming wave in the above bounds is the consequence of the exponent of $|s|$ in the bound of Proposition~\ref{prop:th-domain-derivative}, in combination with Lemma~\ref{lem:RK-CQ}. To reduce these assumptions, estimates that avoid Laplace domain estimates could be attempted, but are beyond the scope of this paper. Approaches in the literature include the combination of time-dependent transmission problems and energy estimates (see \cite[Ch.\@ 7--9]{Say16} and \cite{N24}) and semigroup theory \cite{BR18}.
\end{rem}
\section{Numerical experiments for time-domain inverse scattering}

We address the inverse problem to reconstruct the boundary of a two-dimensional sound-soft scattering object $\Gamma$ from given measurements of the scattered field at spatial observation points $\bfz_\ell \in \Omega$ with $\ell = 1,\dots, M$. 
Let the measured time-discrete signal at the observation points $\bfz_\ell$ and for the times $t_n = \tau n$ with $\tau = T/N_t$, $N_t \in \N$, be denoted by $\bfg \in \ell^2(0,N_t; \R^M)$.

\subsection{Discretization, regularization and a reconstruction algorithm}

For our numerical experiments we restrict our considerations to star-shaped obstacles with some center point $\bfz\in \mathbb R^2$.
We parametrize the boundaries of these objects by cubic periodic splines.
For this purpose, let $Q\in \N$ and let $q_j = 2\pi (j-1) / Q$, $j=1,\dots,Q$.
For some $r_j > 0$, $j=1,\dots,Q$ and $ \bfz \in \R^2$ we define the points
\begin{align}\label{eq:param}
\bfP_j(r_j,\bfz) \, := \, r_j
\begin{bmatrix}
\cos(q_j) \\
\sin(q_j) \\
\end{bmatrix} + \bfz \quad \text{for all } j = 1, \dots, Q \, .
\end{align}
We define the partition
$\tri \, := \, \{q_j \, : \, j = 1,\dots,Q\} \subset [0,2\pi]$
and introduce the set
\begin{align*}
\Pcal_\tri \, := \, \bigg\{ \bfp \in C^2([0,2\pi],\R^2) \, : \, \begin{array}{l}
   \bfp \text{ cubic periodic vector valued spline interpol. } \bfP_j(r_j,\bfz)  \\
    \text{as in } \eqref{eq:param} \text{ for all } j \,,  \; \text{some } r_1,\dots,r_{Q}>0\, \; \text{and } \bfz \in \R^2
  \end{array}
\bigg\} \, .
\end{align*}
By $\Gamma_\bfp$ we denote the boundary of a two-dimensional domain $D$ that is parametrized by $\bfp \in \Pcal_\tri$.
We introduce the relative residual we aim to minimize via
\begin{align*}
f:\Pcal_\tri \to \R \, ,  \quad
f(\bfp) \, := \, \frac{\left\Vert F_{\Gamma_\bfp,\tau}(0) - \bfg \right\Vert_{\ell^2_\tau(0,N_t;\R^M)}^2}{\left\Vert \bfg \right\Vert_{\ell^2_\tau(0,N_t;\R^M)}^2} \, ,
\end{align*}
with the norm defined in \eqref{eq:seq-norm} and the semi-discrete operator $F_{\Gamma_\bfp,\tau}(0)=X_\tau(\Gamma_\bfp)$ from \eqref{eq:F_h-t-tau}.
We introduce two regularization terms $\Psi_1$ and $\Psi_2$ to stabilize the shape reconstruction. The first regularization term $\Psi_1$ penalizes strong curvature and reads
\begin{align*}
\Psi_1:\Pcal_\tri \to \R\, ,  \quad \Psi_1(\bfp) \, := \, \int_0^{2\pi} \kappa^2(\theta) |\bfp'(\theta)| \mathrm d\theta \, ,  \quad \kappa(\theta) \, := \, \frac{p_1'(\theta)p_2''(\theta)-p_2'(\theta)p_1''(\theta)}{|\bfp'(\theta)|^3} \, .
\end{align*}
The term $\kappa$ is the curvature of the curve parametrized by $\bfp \in \Pcal_\tri$ and thus, the term $\Psi_1(\bfp)$ describes the total curvature of $\Gamma_\bfp$.
Moreover, we introduce another penalty term $\Psi_2$, which shall keep the center point of the star domain $\bfz\in \R^2$ close to its geometric center and reads
\begin{align*}
\Psi_2:\Pcal_\tri \to \R\, ,  \quad \Psi_2(\bfp) \, := \, \left|  \bfp_{c} - \bfz\right|^2\, ,
\end{align*}
where $\bfp_{c}$ denotes the geometric center of the domain that $\Gamma_{\bfp}$ encloses.
We add the weighted regularization terms $\alpha_1^2\Psi_1$ and $\alpha_2^2\Psi_2$ to the aim functional $f$, which yields the penalized objective functional
\begin{align}\label{eq:freg}
f_{\text{reg}}: \Pcal_\tri \to \R\, \qquad  f_{\text{reg}}(\bfp) \, := \, f(\bfp) + \alpha_1^2\Psi_1(\bfp) + \alpha_2^2 \Psi_2(\bfp) \, ,
\end{align}
that we minimize using the Gau\ss--Newton method, i.e., we aim to find
\begin{align*}
\bfp^* \, = \, \argmin_{\bfp \in \Pcal_\tri} f_{\text{reg}}(\bfp) \, .
\end{align*}
\begin{rem}[Implementation of the full discretization]
    As a space discretization we use the simple averaging method explained in \cite{DomLuSay14} and \cite[Ch.\@ 7]{HasSay16} (see also the documentation of the deltaBEM package \cite{deltaBEM}). This yields a fully discretized version of the semi-discrete aim functional $f_{\text{reg}}$ from \eqref{eq:freg}.
We direct the reader to \cite[Sec.\@ 5.4]{BanSay22} for comprehensive insights regarding the implementation of Runge--Kutta convolution quadrature methods.
\end{rem}
\begin{rem}[Choosing directions for the Gauß--Newton method]
    For the Gauß--Newton iteration, we require a finite family of perturbations $\bfh_j$ for $j=1,\dots,J<\infty$.  Here, we use a standard choice of perturbations, which varies the interpolation points $\mathbf{P}_j$ of the spline for $j=1,\dots,Q$ in radial direction and complete the directions by shifts of the center $\mathbf z$. A similar reconstruction algorithm based on a Gau\ss--Newton approach is found in \cite{Kirsch93}.
\end{rem}

\subsection{Numerical results}

In this section we delve into numerical examples and explore various scenarios featuring different incident waves and measurement configuration. 
Using both of these, our purpose is to reconstruct the unknown scatterer using the Gau\ss--Newton algorithm.
In order to simulate the forward data $\bfg \in \ell^2(0,N_t; \R^M)$, we employ a 3-stage Radau IIA method and simulate a fully discretized version of \eqref{eq:SV-1-td}
with $N_s = 2\times 10^3$ spatial collocation points to discretize the boundary of the scattering object and $N_t=2\times 10^4$ time steps.
We fix the final time to be $T=20$, what results in the time step size $\tau = 1\times 10^{-3}$.
Note that for this simulation both numbers $N_s$ and $N_t$ are chosen significantly larger than in our reconstructions using the Gau\ss--Newton method.
For our reconstruction algorithm we also use the same Radau IIA method and stop the iteration when the euclidean norm of the update of the scatterer falls below a certain tolerance.
In the plots that accompany the examples, we first visualize the direct problem for three different time points.
The subsequent plots show the initial guess, an intermediate step and lastly, the final approximation of the reconstruction algorithm.
In all of these plots we find the boundary we aim to reconstruct in solid blue and the current iterate of the Gau\ss--Newton method in dashed red.
The positions of the receivers are indicated by black diamonds.
\begin{example}\label{ex:1}
\begin{figure}[t!]
\centering 
\includegraphics[scale=.36]{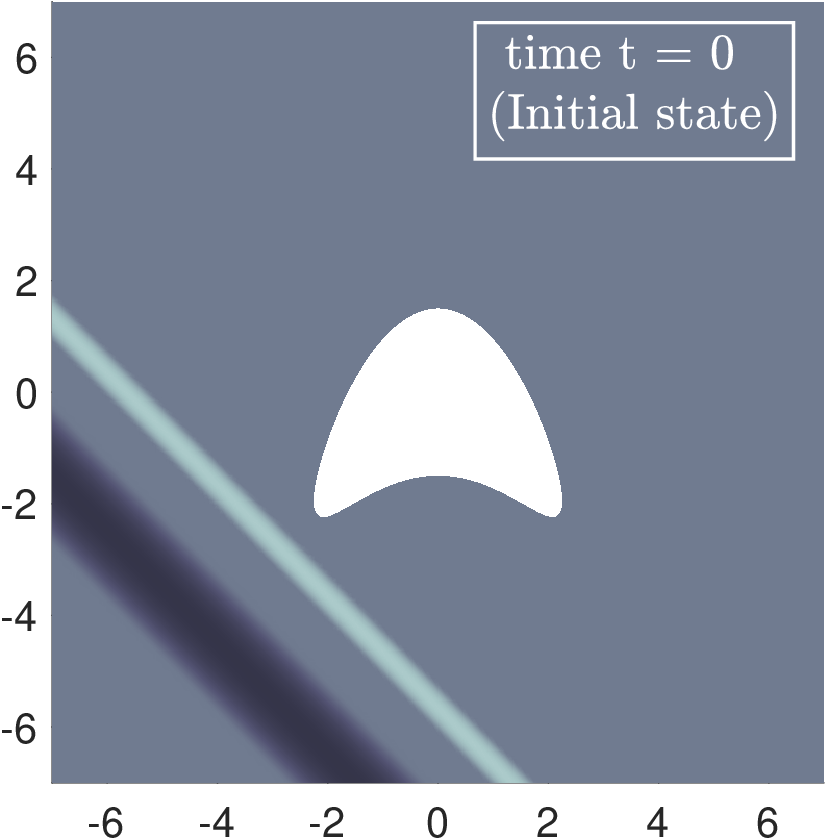} \hfill
\includegraphics[scale=.36]{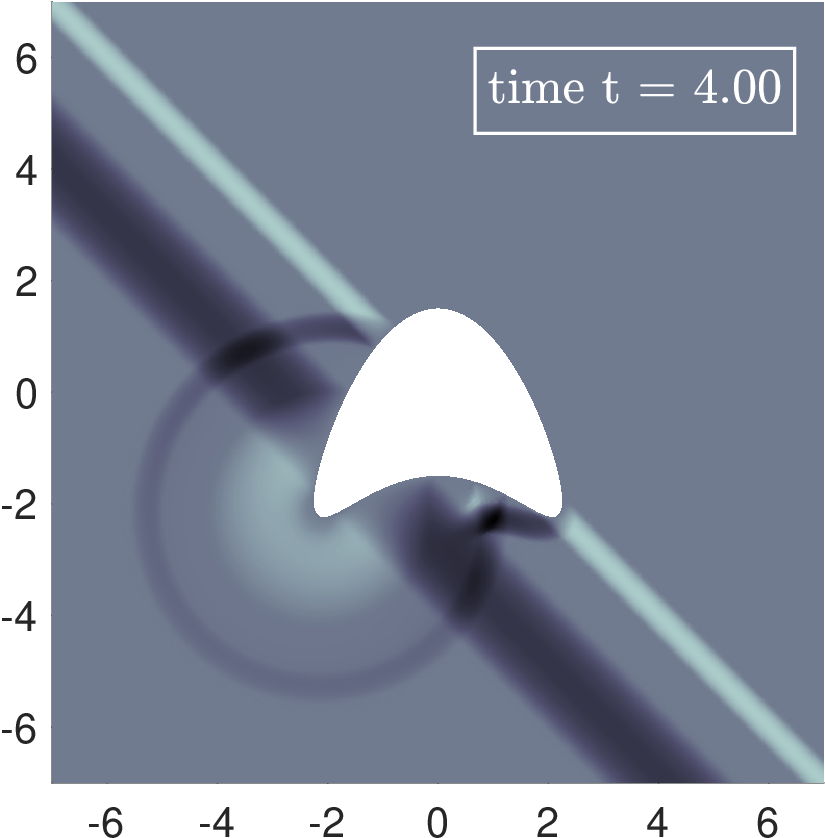} \hfill
\includegraphics[scale=.36]{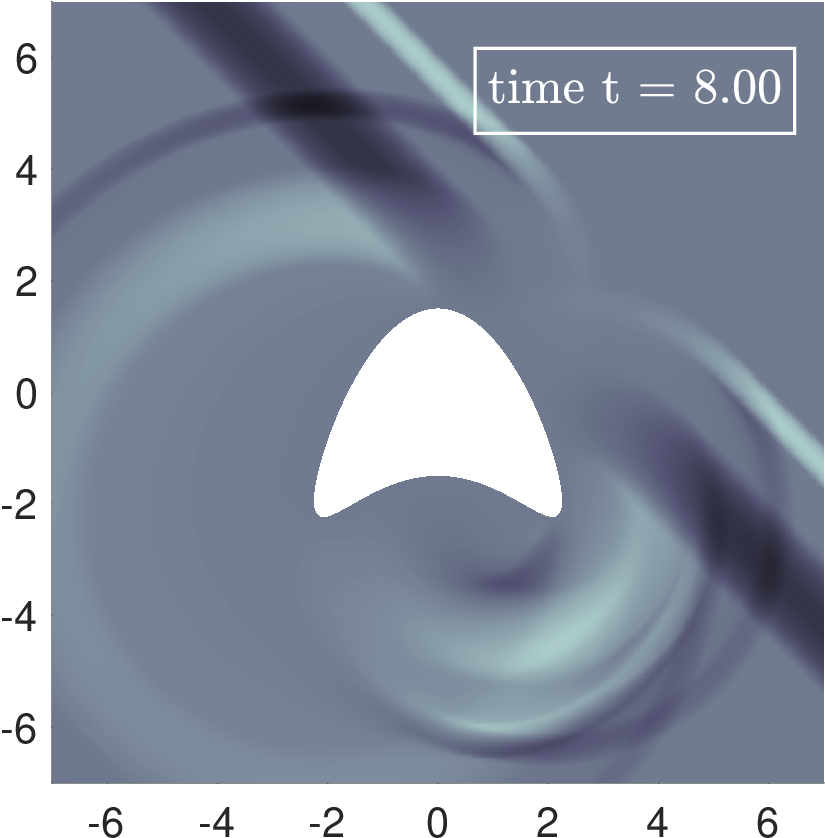}
\caption{Visualization of the scattering object with boundary parametrized by the curve from \eqref{eq:parex1}, together with the total wave $u + u^i$ corresponding to Example~\ref{ex:1} and Example~\ref{ex:2}. 
}
\label{fig:1}
\end{figure}

In our first numerical example let the boundary of the scattering object be given by the curve $\bfp = [p_1,p_2]^\top$ with 
\begin{align}\label{eq:parex1}
p_1(t) \, := \, -(1.5)^2 \sin(t)\, , \quad p_2(t) \, := \, 1.5\left(\cos(t) + 0.65 \cos(2t) - 0.65\right) \, , \; t\in [0,2\pi]\, .
\end{align}
Let $f\in C^\infty_c(\R)$ be defined by $f(t) := \rme^{-1/(1-t^2)}$ for $|t|<1$ and $f(t):=0$ for $|t|\geq 1$.
We define the incident wave $u^i$ by
\begin{align}\label{eq:uiex1}
u^i(\bfx,t) \, := \, f\left(3(\bfx\cdot \bfd - t+4)\right) - f\left(\bfx\cdot \bfd - t+6 \right)\, ,
\end{align}
with $\bfd := 1/\sqrt{2}[1,1]^\top \in S^1$. 
In Figure~\ref{fig:1} we visualize the total wave at three time points, namely at $t=0,4,8$.
For the initial state, i.e.\@, for $t=0$ only the incident wave $u^i(\bfx,0)$ from \eqref{eq:uiex1} is visible.
Then, the wave fronts propagate into the direction $\bfd$ and get scattered by the kite-shaped object.
This is seen in the middle and right plot of Figure~\ref{fig:1}.

We assume measurements of the scattered wave to be available at the spatial points defined by
$\bfz_j := 6[\cos(\theta_{j}), \, \sin(\theta_j)]^\top \in \R^2$ for $\theta_j = (j-1)\pi/4$ and $j = 1,\dots,8$.
In the Gau\ss--Newton method we use $Q=40$ points that are supposed to be interpolated by a cubic periodic spline.
We pick $N_s=1\times10^3$ points to evaluate this spline and use $N_t=5 \times 10^3$ points in time. This corresponds to the time step size $\tau = 4\times 10^{-3}$.
Furthermore, we
choose the regularization parameters $\alpha_1 = 0.02$ and $\alpha_2 = 0.5$.

We initialize the iteration with a circular initial guess having the radius $r=0.5$ and the center point at $\bfz = [-1,\; -1.5]^\top$.
Three snapshots from the convergence history can be found in Figure~\ref{fig:2}.
We find that the iteration stops after 17 steps, with the final approximation found in the right plot of Figure~\ref{fig:2}.
\begin{figure}[t!]
\centering 
\includegraphics[scale=.36]{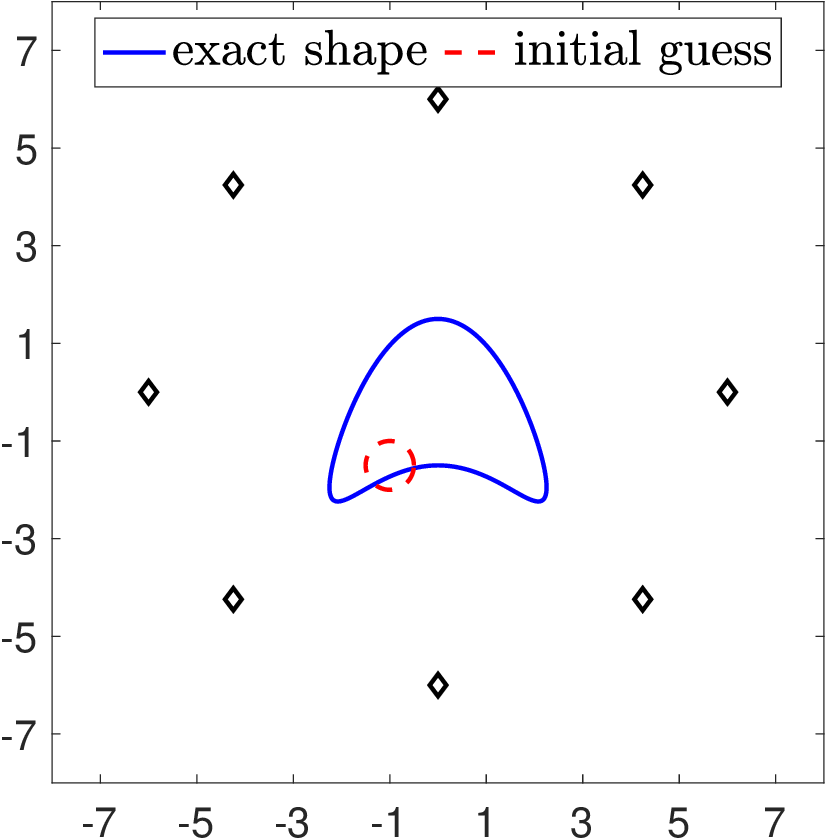}\hfill
\includegraphics[scale=.36]{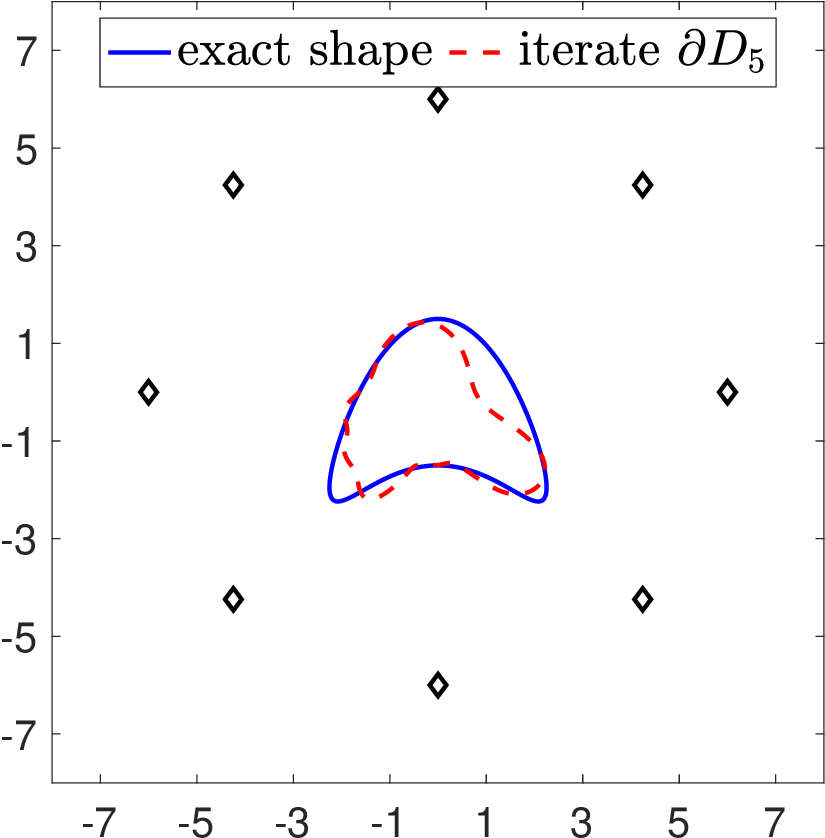} \hfill
\includegraphics[scale=.36]{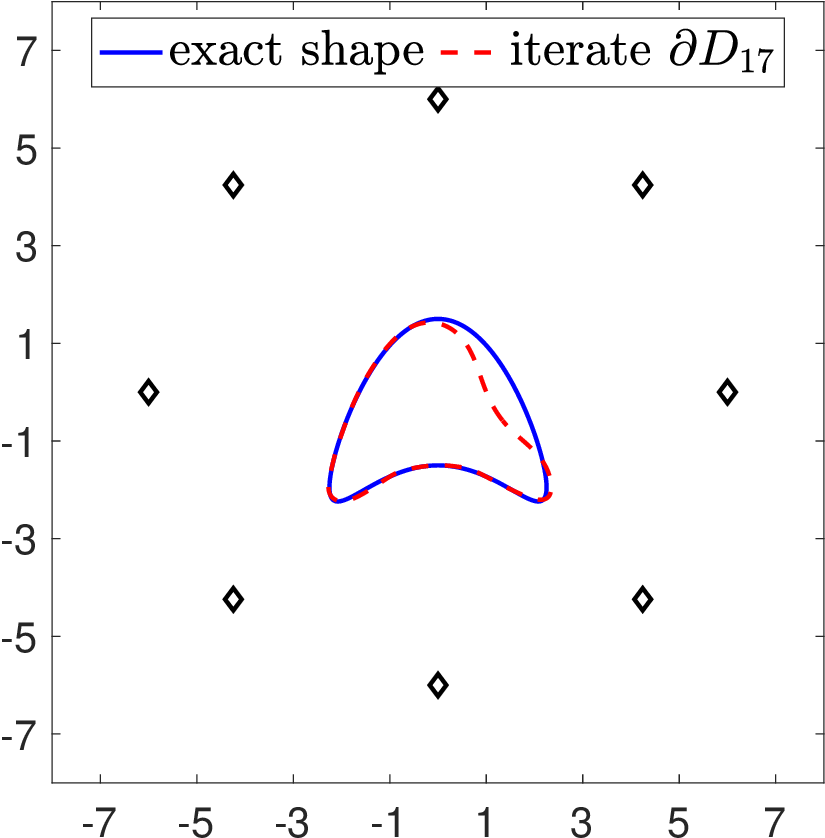}
\caption{Convergence history starting with a circle with radius $r=0.5$ and center point $\bfz = [-1, \; -1.5]^\top\in \R^2$ (left). The middle plot shows the iterate $\ell = 5$. The right plot shows the final result after 17 steps. }
\label{fig:2}
\end{figure}
The observation points $\bfz_j$, $j=1,\dots,8$ are 
located all around the scattering object and thus, only
the top-right contour - the part of the boundary that lies in the shadow of the incident wave - is somewhat off from the exact shape.
\begin{rem}
Using the geometrical setup from Example \ref{ex:1} and using an additional incident wave from the opposite angle may substantially improve the reconstruction quality.
\end{rem}
\end{example}
\begin{example}\label{ex:2}
In our second example, we study the same configuration for the unknown scattering object as in Example~\ref{ex:1} together with an incident wave $u^i$ of the form \eqref{eq:uiex1} with $\bfd := 1/\sqrt{2}[1,1]^\top \in S^1$. The forward scattering problem is also depicted by Figure~\ref{fig:1}. 
This time, we assume $M=4$ measurement points to be given by $\bfz_j := 6[\cos(\theta_{j}), \, \sin(\theta_j)]^\top \in \R^2$ for $\theta_j = \pi/4(j+3)$ and $j = 1,\dots,4$, i.e., the measurement points are distributed on the circular arc around the origin defined by the radius 6 and angle between $\pi$ and $7/4\pi$.
We start {the reconstruction algorithm} with the same initial guess as in Example~\ref{ex:1}, i.e.\@ we use the disc centered in $\bfz = [-1, \, -1.5]^\top$ having the radius $r=0.5$.
Moreover, we pick the regularization parameters $\alpha_1 = 0.06$ and $\alpha_2 = 0.5$.
The convergence history can be found in Figure~\ref{fig:3}. The reconstruction algorithm stops after 20 iterations
\begin{figure}[t!]
\centering 
\includegraphics[scale=.36]{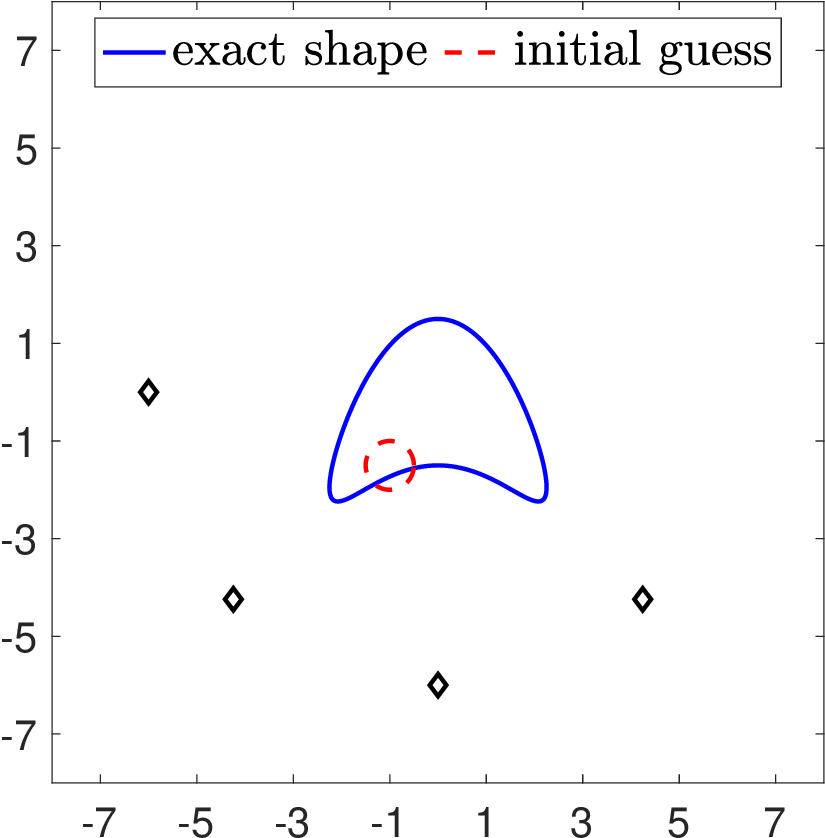}\hfill
\includegraphics[scale=.36]{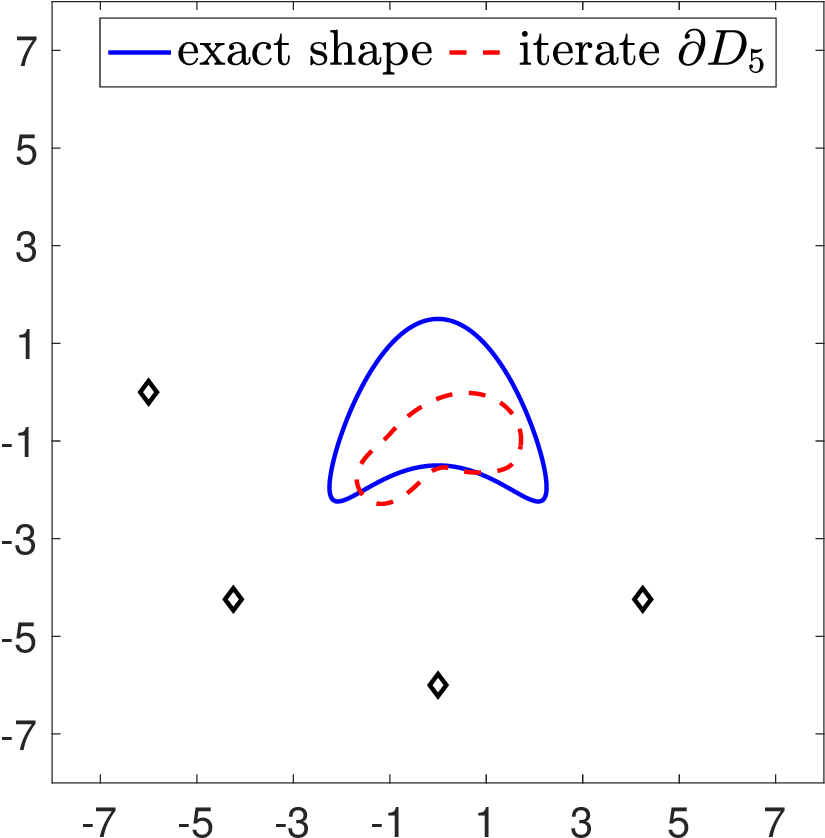}  \hfill
\includegraphics[scale=.36]{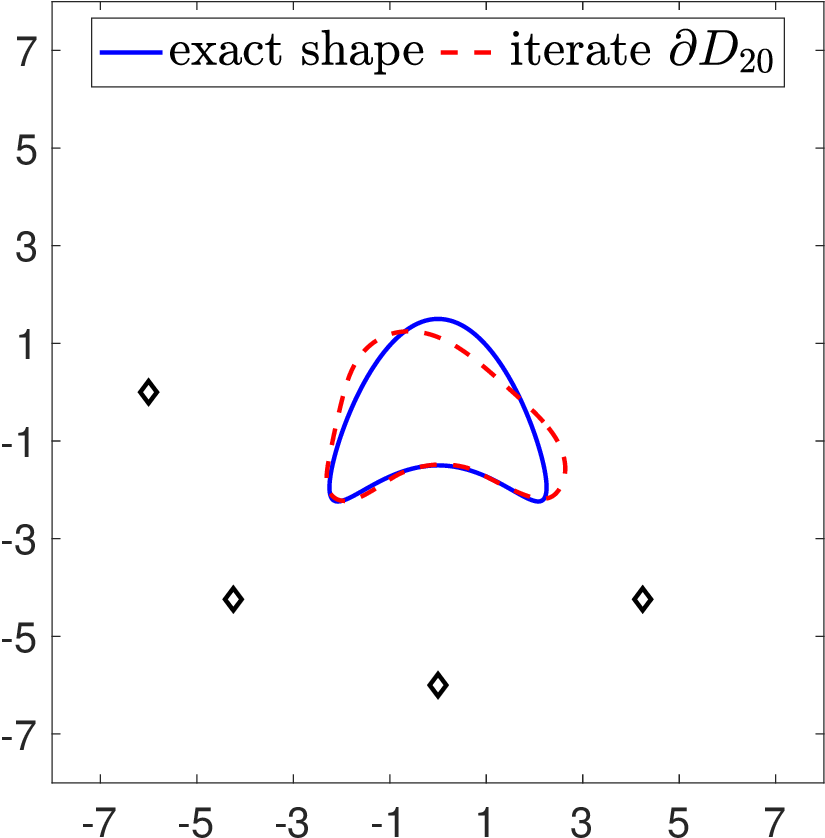}
\caption{Convergence history starting with a circle with radius $r=0.5$ and center point $\bfz = [-1, \; -1.5]^\top\in \R^2$ (left). The middle plot shows the iterate $\ell = 5$. The right plot shows the final result after 20 steps. }
\label{fig:3}
\end{figure}
and yields a good approximation of the bottom and bottom-left contour of the scattering object, which is the main contact surface of the wave front of the incident wave $u^i$.
We observed similar phenomena for different configurations of incident waves and geometrical setups.
The gathered backscattered data appears to be essential for a reasonable reconstruction of the surface, on which the incident wave impinges. Receivers that lie in the shadow of the obstacle do not seem to contribute substantially to an effective reconstruction. 
\end{example}

\begin{example}\label{ex:3}
In our third example we study a more complicated shape of the scattering object.
We define some points in the two-dimensional space, interpolate them using a cubic periodic spline and obtain a dove-shaped scatterer (see Figure \ref{fig:4}).
In this particular scenario, we study an incident wave formed by a superposition of emissions originated from source points located away from the scattering object, which also act as receivers. These emitters and receivers are denoted by $\bfz_j$ for $j=1,\dots,M$. Let $f(t) := \rme^{-1/(1-4(t-1)^2)}$ for $|t-1|<1/2$ and $f(t):=0$ for $|t-1|\geq 1/2$. We define the incoming wave by (see \cite[Eq.\@ 13]{Rie23})
\begin{align}\label{eq:uisource}
u^i(\bfx,t) \, = \, \sum_{j=1}^M \mathcal{H}\left( {t}{| \bfx-\bfz_j |^{-1}}-1\right) \int_0^{\acosh\bigl( \frac{t}{| \bfx-\bfz_j |}\bigr)} f\left( t-| \bfx - \bfz_j | \cosh(\theta) \right) \dtheta \, ,
\end{align}
where $\mathcal{H}$ denotes the Heaviside function. In our simulation, we choose $M=5$ and set the points as visualized in Figure~\ref{fig:5}.
In practice, we approximate the integral in \eqref{eq:uisource} by the built-in Matlab routine \textit{integral}. 
\begin{figure}[t!]
\centering 
\includegraphics[scale=.36]{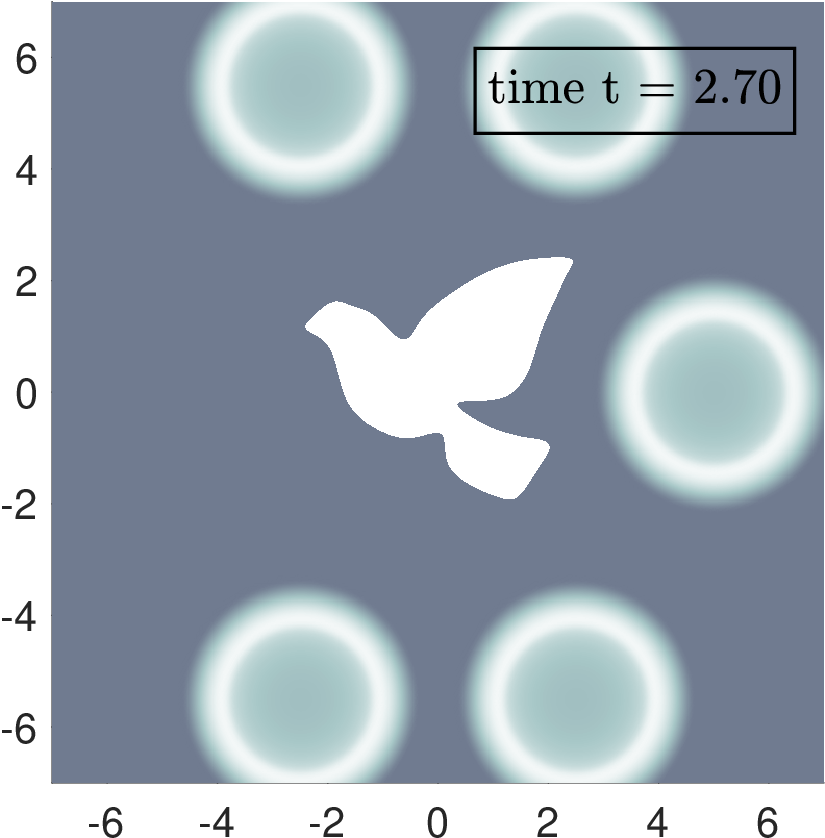} \hfill
\includegraphics[scale=.36]{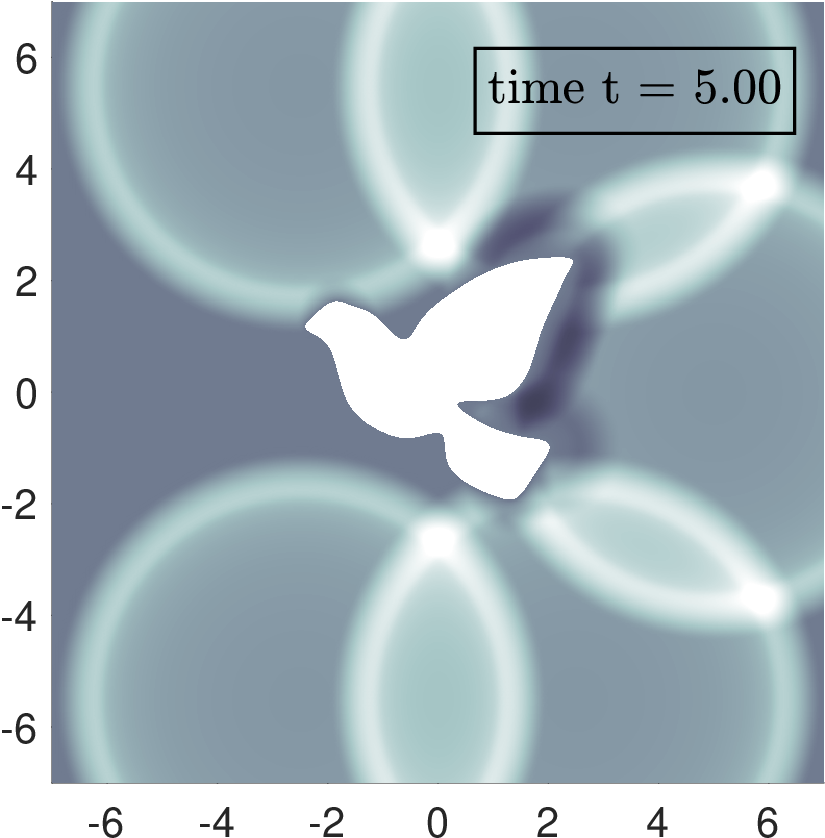} \hfill
\includegraphics[scale=.36]{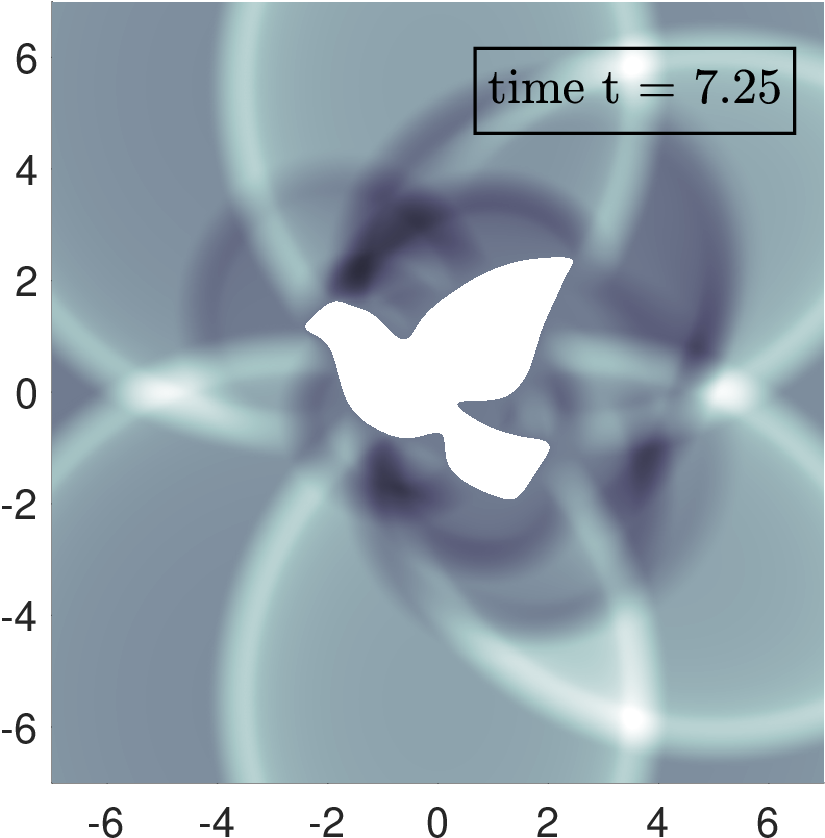}
\caption{Visualization of the dove-shaped scattering object together with the total wave $u + u^i$ of Example~\ref{ex:3} at different times. The incoming wave is zero for $t\leq 0.5$.}
\label{fig:4}
\end{figure}
In Figure~\ref{fig:4} we visualize the total wave at three time points, namely at $t=2.7,5,7.25$.
In the left plot only the incident wave $u^i(\bfx,2.7)$ from \eqref{eq:uisource} is visible.

In our reconstruction algorithm, we use $Q=40$ points, which we interpolate by a cubic periodic spline.
This spline is evaluated at $N_s=1\times 10^3$ spatial points, with $N_t=5\times 10^3$ time steps until the final time $T=20$.
We set the regularization parameters to $\alpha_1 = 0.02$ and $\alpha_2 = 0.06$.
In Figure \ref{fig:5} we find the results.
The algorithm stops after 16 steps with a final approximation that captures the overall shape of true scatterer. 
However, it is worth noting that the fine details from the actual shape of the scatterer are not reconstructed by the algorithm.
In view of the fact that the algorithm only supports star-shaped objects, we believe that the reconstruction is still reasonable.
\end{example}
We conclude this section with a remark about noisy data.
\begin{figure}[t!]
\centering 
\includegraphics[scale=.36]{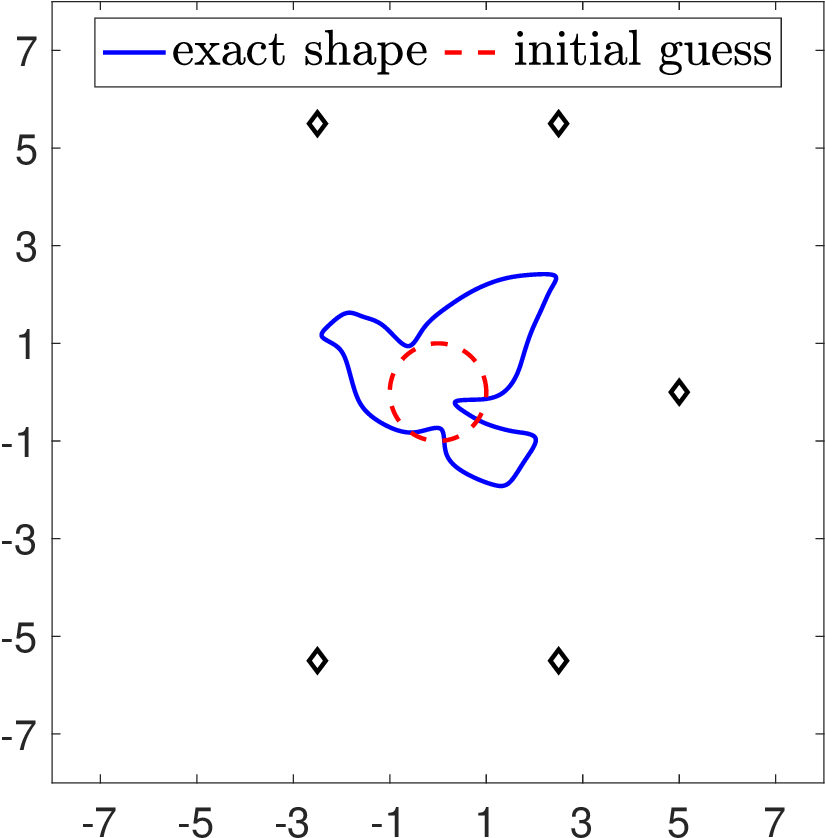}\hfill
\includegraphics[scale=.36]{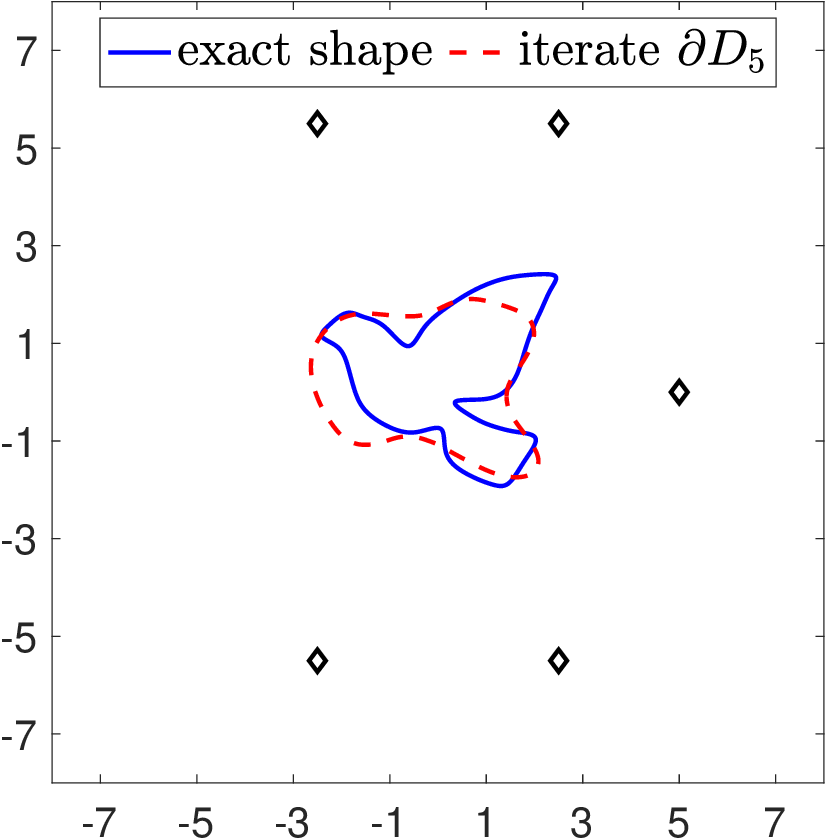} \hfill
\includegraphics[scale=.36]{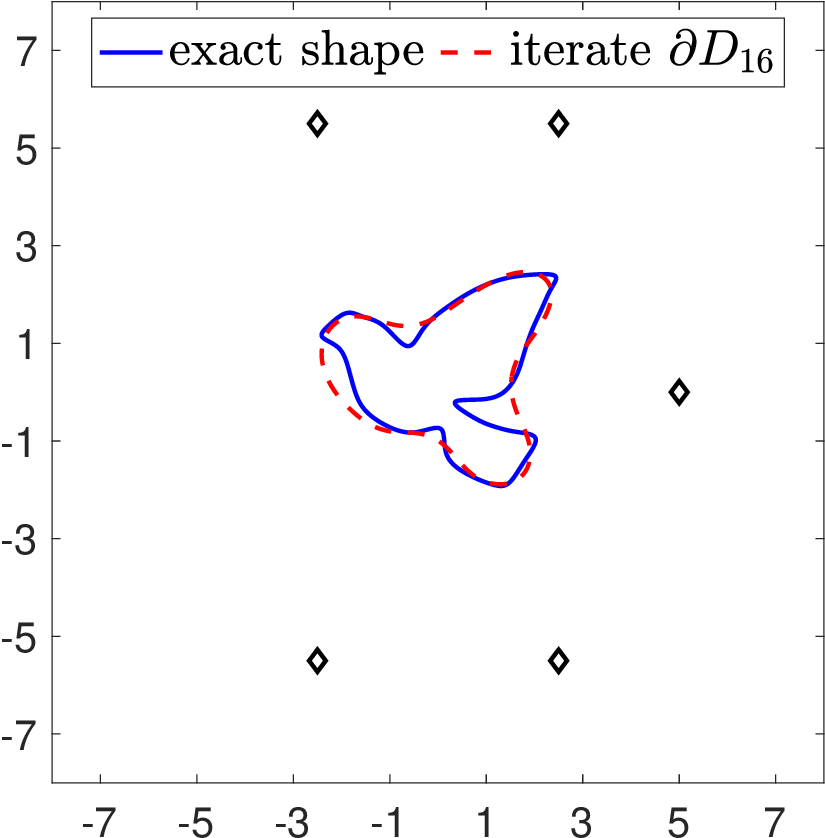}
\caption{Convergence history starting with a circle with radius $r=1$ and center point $\bfz = [0, \; 0]^\top\in \R^2$ (left). The middle plot shows the iterate $\ell = 5$. The right plot shows the final result after 16 steps.}
\label{fig:5}
\end{figure}
\begin{rem} 
The algorithm appears to be rather resistant to additional noise. 
For given data $\bfg\in \ell^2(0,N_t;\R^M)$ and a uniformly distributed random vector $\bfzeta \in \ell^2(0,N_t;\R^M)$ we simulated noisy data $\bfg^\delta\in \ell^2(0,N_t;\R^M)$ via
\begin{align*}
\bfg^\delta \, := \, \bfg + \delta \bfzeta \frac{\Vert \bfg \Vert_{\ell_\tau^2(0,N_t;\R^M)}}{\Vert \bfzeta \Vert_{\ell_\tau^2(0,N_t;\R^M)}}
\end{align*}
for $\delta = 0.3$. This corresponds to $30\%$ additional relative noise.
Then, we considered this data as the given measurement of the scattered wave and started the reconstruction algorithm.
Restarting Example \ref{ex:1}, Example \ref{ex:2} and Example \ref{ex:3} with this noisy data yields results, which are qualitatively the same as those that are found in Figure \ref{fig:2}, Figure \ref{fig:3} and Figure \ref{fig:5}. 
\end{rem}

\section*{Acknowledgment}
Funded by the Deutsche Forschungsgemeinschaft (DFG, German Research Foundation) - Project-ID 258734477 - SFB 1173. 

We would like to thank Roland Griesmaier and Christian Lubich for their remarks on improving this manuscript and for fruitful discussions. Furthermore, we thank Eliane Kummer and the anonymous referees for their valuable feedback and helpful comments.

\small
\bibliographystyle{abbrvurl}
\bibliography{references}
\end{document}